\documentclass[11pt]{article}
\usepackage[square, numbers]{natbib}

\usepackage{amsmath,amssymb, enumitem}

\usepackage[left=0.8in,top=0.8in,right=0.8in,bottom=0.8in]{geometry}
\usepackage{graphics,epsfig}

\usepackage{url}

\begin{document}

\newenvironment {proof}{{\noindent\bf Proof.}}{\hfill $\Box$ \medskip}

\newtheorem{theorem}{Theorem}[section]
\newtheorem{lemma}[theorem]{Lemma}
\newtheorem{condition}[theorem]{Condition}
\newtheorem{definition}[theorem]{Definition}
\newtheorem{proposition}[theorem]{Proposition}
\newtheorem{remark}[theorem]{Remark} 
\newtheorem{hypothesis}[theorem]{Hypothesis}
\newtheorem{corollary}[theorem]{Corollary}
\newtheorem{example}[theorem]{Example}

\renewcommand {\theequation}{\arabic{section}.\arabic{equation}}
\def \non{{\nonumber}}
\newcounter{fig}
\newcommand{\ddt}{\frac{d}{dt}}
\newcommand{\lskip}{\phantom{.}}

\newcommand{\f}{\frac}

\newcommand{\Rt}{\longrightarrow}
\newcommand{\rt}{\rightarrow}
\newcommand{\RT}{\Rightarrow}
\newcommand{\Lt}{\longleftarrow}
\newcommand{\lt}{\leftarrow}
\newcommand{\LT}{\Leftarrow}
\newcommand{\LRT}{\Leftrightarrow}

\newcommand{\np}{\noindent}
\newcommand{\T}{\text}
\newcommand{\hs}{\hspace}
\newcommand{\vs}{\vspace}

\newcommand{\w}{\omega}
\newcommand{\W}{\Omega}
\newcommand{\g}{\gamma}
\newcommand{\G}{\Gamma}
\newcommand{\e}{\epsilon}
\newcommand{\oo}{\infty}
\newcommand{\s}{\sigma}
\def\SC{\mathcal}
\def\ub{\underbar}
\def \triple|{|\! | \! |}

\def\le{\left}
\def\ri{\right}

\def\l{\lambda}
\def\L{\Lambda}
\def\ot{\otimes}
\def\ott{\hat{\otimes}_\pi}
\def\oto{\hat{\otimes}_{op}}
\def\oth{\hat{\ot}_{HS}}
\def\<{\langle}
\def\>{\rangle}
\def\~{\tilde}
\newcommand{\EE}{\ensuremath{\mathsf{E}}}
\newcommand{\PP}{\ensuremath{\mathsf{P}}}

\def\y{\mathbf y^*}
\def\N{\mathbb N}
\def\Z{\mathbb Z}
\def\R{\mathbb R}
\def\C{\mathbb C}
\def\H{\mathbb H}
\def\LL{\mathbb L}
\def\K{\mathbb K}
\def\X{\mathbb X}
\def\Y{\mathbb Y}
\def\U{\mathbb U}
\def\E{\mathbb E}


\title{\Large\ {\bf Large deviation principle for stochastic integrals and stochastic differential equations driven by infinite-dimensional semimartingales  }}

\author{Arnab Ganguly\thanks{Research supported in part by NSF grants DMS 05-53687, 08-05793 and  Louisiana Board of Regents through the Board of Regents Support Fund (contract number: LEQSF(2016-19)-RD-A-04).}\\
Department of Mathematics\\
Louisiana State University\\
aganguly@lsu.edu} 
\date{}
\maketitle
\begin{abstract}
\noindent
The paper concerns itself with establishing large deviation principles for a sequence of stochastic integrals and stochastic differential equations driven by general semimartingales in  infinite-dimensional settings. The class of semimartingales considered is broad enough to cover Banach space-valued semimartingales and the martingale random measures. Simple usable  expressions for the associated rate functions are given in this abstract setup.  As illustrated through several concrete examples, the results presented here provide a new systematic approach to the study of large deviation principles for a sequence of Markov processes. \\


\end{abstract}

\setcounter{equation}{0}
\section{Introduction}
The theory of large deviations is roughly the study of the exponential decay of the probability measures of certain kinds of extreme or tail events (see \cite{DZ98} for  some general principles of this theory).  More precisely, as formulated by Varadhan \cite{Var66}, the large deviation principle for a sequence of probability measures  $\left\{\mu_n\right\}$ is defined as follows:
\begin{definition}
Let $U$ be a Polish space and $\left\{\mu_n\right\}$ a sequence of probability measures on $(U, \SC{U})$, where $\SC{U}$ is the Borel $\sigma$-algebra on $U$. $\left\{\mu_n\right\}$ satisfies a {\bf large deviation principle} (LDP) with rate function $I: U \rt [0,\infty)$, if 
\begin{align*}
 \liminf_{n\rt \infty} \f{1}{n} \log \mu_n (O) \geq -I(O), \ \ \mbox{for every open set } \ O \in \SC{U},\\
\end{align*}
and
\begin{align*}
 \limsup_{n\rt \infty} \f{1}{n} \log \mu_n (C)\leq -I(C), \ \ \mbox{for every closed set } \ C \in \SC{U}.\\
\end{align*}
Here for a set $A$, $I(A) = \inf_{x\in A}I(x)$.
 \end{definition}
The rate function $I$ is generally taken to be lower semicontinuous, and under that condition it is unique. Much of the earlier work on the large deviation principle (Donsker and Varadhan \cite{DV75, DV76}) was based on change of measure techniques, where a new measure is identified under which the events of interest have high probability, and then the probability of that event under the original probability measure is calculated using the Radon-Nikodym derivative.

Starting with the pioneering works of Fredlin and Wentzell \cite{FW98}, there has been a vast body of work on large deviation asymptotics for various small noise stochastic differential equations (SDEs) driven by  Brownian motions in both the finite-dimensional and infinite-dimensional settings, and a partial list of references is \cite{Ra80, MiNuSa92, MeMi98, Sun06, BDV08, YaHo08, SunSri09, Bes09, DuMi09, BDM10, Rock10, Liu10}. Historically, there are comparatively less amount of results  available on large deviations of SDEs with jumps, particularly, for infinite-dimensional models. However, in the past decade quite a few papers have come up where the authors proved large deviation results for  SDEs driven by Poisson random measures (for example, see \cite{LipPuha92, RoZha07, BCD13}). 

 Traditional route for proving LDP involves discretizing the given SDE and proving an LDP first for the simplified discretized system usually through a contraction principle. One then shows  that the original system is {\em exponentially close} to the discretized system through some technical probability estimates and argues that a LDP holds for the original model. Although this program of proving LDP works for many finite-dimensional models, it is considerably difficult to carry it out for infinite-dimensional models. One reason behind that is the fact that discretization needs to be done for both the space and the time variables. The presence of jumps in these infinite-dimensional models makes such a program even more difficult to implement as several key estimates are exceedingly difficult to obtain in these cases and furthermore have to be done on a case by case basis. Also, identification of the rate function in a suitable usable form remains as the subsequent challenging task. This probably explains why there haven't been too many works on infinite dimensional models with jumps until recently.
 
 The main aim of the present paper is to investigate  large deviation principles for infinite-dimensional SDEs driven by general semimartingales. Toward this end, the paper explores the ideal growth conditions one needs on the driving semimartingales for large deviation-results for the corresponding SDEs. More specifically, conditions are sought which would ensure that if  $ \left\{(X_n(0), Y_n)\right\}$  satisfies a LDP, then a LDP also holds for $\left\{X_n\right\}$, where $\left\{X_n\right\}$ solves
$$X_n = X_n(0) + F(X_{n-}) \cdot Y_n.$$
Here for a semimartingale $Y$ and a cadlag adapted process $X$, $X_-\cdot Y(t)$ denotes the stochastic integral $\int_0^t X(s-)dY(s)$. The natural first step in this approach is to conduct a large deviation analysis for stochastic integrals, that is, we investigate conditions required on the sequence $\left\{Y_n\right\}$  which  guarantees a LDP for $\left\{X_{n-}\cdot Y_n\right\}$ whenever a LDP holds for the pair $\left\{(X_{n},Y_n)\right\}$. Since in general, there does not exist a continuous function $f$ such that   $X_{-}\cdot Y= f(X,Y)$, the result cannot be arrived at by a simple application of the contraction principle!

For finite-dimensional  processes, a {\em uniform exponential tightness} (UET) condition on the sequence $\left\{Y_n\right\}$ was given in Garcia \cite{Gar08} which yields the desired result. The idea of  the UET condition is inspired by the {\em uniform tightness} condition used by Jakubowski, Me{\'m}in and Pag{\`e}s \cite{JMP89} to prove  weak convergence results for a sequence of stochastic integrals (also see \cite{KP91, KP96_1}).  Roughly speaking, the UET condition says that as long as an integrand sequence remains bounded, the probability of the magnitudes of the sequence of corresponding integrals with respect to $Y_n$ becoming unbounded is exponentially small. Under such a growth condition, Garcia (\cite[Theorem 1.2]{Gar08}) proved the following result:
\begin{theorem}
\label{Gar_th}
Let $\left\{Y_n\right\}$ be a uniformly exponentially tight sequence of $\left\{\SC{F}^n_t\right\}$-adapted real-valued semimartingales and $\left\{X_n\right\}$ be a sequence of  $\left\{\SC{F}^n_t\right\}$-adapted  real-valued cadlag processes. If $\left\{(X_n, Y_n)\right\}$ satisfies a large deviation principle with a rate function $I$, then so does the tuple $\left\{(X_n,Y_n, X_{n-}\cdot Y_n)\right\}$ with the rate function $J$ given by
\begin{eqnarray}\label{ratefunct_fin}
J(x, y, z)& = 
\begin{cases}
I(x,y), & z = x\cdot y, \  y \mbox{ finite variation},\\
\infty, & \mbox{otherwise.}
\end{cases}
\end{eqnarray}
Here $x\cdot y(t) \equiv \lim_{\|\sigma\|\rt 0} \sum_{i}x(t_i)(y(t_{i+1}) - y(t_i))$, where $\sigma = \{0=t_0<t_1<\hdots,t_n=t\}$ is a partition of the interval $[0,t)$ and $\|\sigma\| \equiv \max_i(t_{i} - t_{i-1})$ is the mesh of the partition $\sigma$.
\end{theorem}
A similar result for stochastic differential equations in the finite-dimensional setting has also been proved. The feature which actually makes the above theorem unique and interesting is the following: although the result covers general stochastic integrals (and not limited to the ones driven by usual integrators like Brownian motions or Poisson processes), the form of the rate function of the stochastic integrals is quite simple in the sense that it is easily expressible in terms of the known rate function of the original process-sequence $\{(X_n,Y_n)\}$. Note that the rate function $J$ can only be finite for those paths for which the $y$-parts are of finite variation.

Our goal is of course to investigate such general results in infinite-dimensional settings, and moreover, we are quite ambitious in our goal in the sense that we want to prove our results for quite a general class of  infinite-dimensional semimartingale integrators rather than for a smaller class which are restricted to taking their values in a specific type of infinite-dimensional space. In other words, we want to consider a class of semimartingales which is broader than, for example, the class of Hilbert space-valued semimartingales. The motivation behind considering such a broad class is that most of the popular integrators take values in different kinds of infinite-dimensional spaces --- for example, space-time Gaussian white noise take values in the space of distributions, a Wiener process with trace-class covariance operator takes values in a Hilbert space while a Poisson random measure takes values in the space of counting measures. The class of semimartingales we find suitable for this purpose is the class of $\H^\#$-semimartingales which was introduced by Kurtz and Protter in \cite{KP96_2} (see Section \ref{chap:Hsharp}). Here $\H$ is a separable Banach space, and an $\H^\#$-semimartingale $Y$ can be thought of as a semimartingale indexed by elements of  $\H$ and time satisfying some necessary properties one needs to do stochastic analysis with it. But as is made clear in Section \ref{chap:Hsharp}, an $\H^\#$-semimartingale $Y$ might not take values in $\H$ or some other Banach space! Indeed, one of the advantages of working with $\H^\#$-semimartingales is that the knowledge of specific path-space where the sequence of $\H^\#$-semimartingales take values is not important; as is illustrated in subsequent sections all the necessary definitions, conditions and results can be formulated through a collection of finite-dimensional projections of these semimartingales. Furthermore, many of the necessary conditions for the LDP are encoded in the choice of the indexing space $\H.$ The stochastic integral process $X_-\cdot Y$, however, is assumed to $\LL$-valued for some Banach space $\LL$. 

We now make some comments about mathematical technicalities in the paper.  We first extend the concept of UET-condition to a sequence of $\H^\#$-semimartingales (c.f Section \ref{sec:UET}). The large deviation results for infinite dimensional stochastic integrals and SDEs were proven through an approach that is analogous to the Prohorov compactness approach to weak convergence. This approach has its roots in the works  Puhalskii \cite{AP93}, O'Brien and Vervaat \cite{VB95}, de Acosta \cite{Acosta97}. The proof of weak convergence typically involves verification of {\em tightness} of the sequence. A similar role is played by the {\em exponential tightness} (Definition \ref{def:exptight}) condition in the `weak convergence approach' to  large deviation theory. Puhalskii \cite{AP93} (and in more general settings,  O'Brien and Vervaat \cite{VB95} and de Acosta \cite{Acosta97}) showed that exponential tightness implies existence of a large deviation principle along a subsequence (see Theorem 3.7 of \cite{FK06}). Verification of exponential tightness for the necessary cadlag processes in the present paper utilizes some useful results established in Feng and Kurtz \cite{FK06}. These results are in the same spirit as those standard results on tightness of a sequence of processes involving estimation of their fluctuations. LDP for the stochastic integrals driven by these infinite-dimensional $\H^\#$-semimartingales was established by using appropriate finite-dimensional projections and proving that these finite-dimensional approximations are sufficiently close so that an approximation result like \cite[Lemma 3.14]{FK06}
can be employed to establish the desired LDP  (c.f Theorem \ref{LDP_mainth}). However, it does not give the corresponding rate function in a usable form. Indeed, as evident from Theorem \ref{LDP_mainth}, the rate function coming from the use of \cite[Lemma 3.14]{FK06} is quite complicated.

Although such approximation analyses for proving LDP  have to be done carefully for the broad type of infinite dimensional problem considered here, the major mathematical challenge, however, is to express the rate function in a simple usable form as in \eqref{ratefunct_fin}, and a significant portion of the paper is devoted toward that end. In doing so in the infinite-dimensional settings, one of the major obstacles we face is that the path-space of the infinite-dimensional $\H^\#$-semimartingales is not known --- as mentioned, it depends on specific examples. But one of the significant mathematical accomplishments of the present paper lies in  demonstrating that under the UET-condition on $\{Y_n\}$, the joint rate function $J(x,y,z)$ of $\{(X_n, Y_n, X_{n-}\cdot Y_n)\}$ can only be finite on those paths $(x,y,z)$ whose $y$-parts, in some suitable sense, are equivalent to paths in $\H^*$ having bounded total variation (c.f Theorem \ref{UETprop}). This allows us to define a Riemann-type integral with respect to $y$, and the rate function $J$ can be expressed in a form similar to \eqref{ratefunct_fin} in the infinite-dimensional settings (see Theorem \ref{th_ldp_si} for the result on stochastic integrals and Theorems \ref{sde} and \ref{sde3} for the corresponding result on SDEs). This has been achieved by using tools from basis theory of separable Banach spaces, in particular, a proper pseudo-basis in $\H$ is identified which enables us to interpret $y$ as paths in $\H^*$. A basis or more generally a pseudo-basis in a Banach space is an extension of the concept of complete orthonormal set for a separable Hilbert space.

By investigating the large deviation asymptotics in such a general abstract setting, the paper not only provides deeper insight into the problem's fundamental structure, but also pave the way for a systematic program to LDP of Markov processes. Many Markov processes can be represented as solutions of stochastic differential equations driven by various types of finite or infinite-dimensional semimartingales, where a LDP for the sequence of driving semimartingales comes from standard textbook type results. Then one uses Theorems \ref{sde} or \ref{sde3} to get the LDP for the desired sequence of processes relatively easily after verifying the required conditions. In particular, use of these general results enables one to avoid any complicated discretization or approximation schemes that are constructed on a case by case basis to prove LDP for different infinite-dimensional SDEs. The concrete examples at the end demonstrate such applications of our results, thereby, illustrating the usefulness of this kind of a unified approach toward LDP problems. 

Before outlining the organization of the paper, we mention two other alternative approaches toward LDP for Markov processes, and depending on the problem the user might give preference to one over the others. Both of these approaches make use of the connection between control theory and large deviations in their own different ways.  Connections between control theory and large deviations can be traced back to the works of Fleming \cite{WF78, WF85}. The ideas there were then extended further by  Feng and Kurtz in \cite{FK06} who used convergence of the corresponding sequence of nonlinear semigroups associated with a sequence of Markov processes to prove the desired large deviation results. This was achieved by studying the convergence of the corresponding generators to a limiting operator $H$ and then verifying a comparison principle for viscosity solutions of an infinite-dimensional Hamilton Jacobi equation associated with $H$. 
A variational representation of $H$ is then constructed, and the limiting semigroup is subsequently identified as the so called {\em Nisio semigroup} associated with an optimal control problem. This control problem then gives an explicit and more usable representation of the rate function. Although this program of verification of LDP has been carried out for lots of models in \cite{FK06}, the proofs, particularly, the verification of a comparison principle, are often quite technical borrowing heavily from PDE theory  and need to be done on a case by case basis. A different approach based on variational representations of  certain exponential functionals of the integrators was developed in the works of Budhiraja, Dupuis, etal. \cite{BoDu98, BDV08, BDM11} (also see \cite{DE97} for a treatment of discrete processes). Establishing LDP involves studying the asymptotics of the sequence $n^{-1}\log E[\exp(-nF(X_n))]$ and through variational representation these quantities can be interpreted as costs of an optimal control problem. From there it can be argued that the main step in proving a LDP for $\{X_n\}$ entails studying weak convergence of a controlled version of the sequence $\{X_n\}$ given by solutions of controlled perturbations of the original sequence of SDEs. The diffculty level of studying the relative compactness properties of the corresponding controlled processes varies from systems to systems --- while it is comparatively   `easier' for SDEs driven by Brownian motions \cite{BoDu98}, the analysis is more intricate for systems with jumps \cite{BCD13, BDG14} and more so in infinite-dimensional settings.

The paper has been written with the effort and intention to make it largely self-contained for probabilists. Hence, many of the concepts and results, particularly the ones from functional analysis, which might not be familiar to a general researcher in probability theory, have been briefly described. The rest of the paper is organized as follows.  A brief introduction to $\H^\#$-semimartingales is given in Section \ref{chap:Hsharp}. Section \ref{sec:UET} concerns itself with the idea of uniform exponential tightness.  Important  results on exponential tightness and large deviations of stochastic integrals are established in Section \ref{sec:ldp}. Identification of the rate function of the integrals in a simple form was the biggest challenge in this infinite-dimensional setting, and this is the focus of Section \ref{sec:rateid}. The large deviation results on general SDEs are proved in Section \ref{sec:SDE}. Section \ref{sec:ex} then illustrates a systematic program for verification of LDP through several concrete examples. Finally, the Appendix collects some key results and concepts including those from Orlicz spaces and basis theory of Banach spaces.\\

\np
{\bf Notations:} Unless otherwise specified, $\H, \K$ will denote generic separable Banach spaces. $\H^*_c$ will denote the dual space $\H^*$ equipped with the topology of uniform convergence on compacts. For a complete and separable metric space $(\E, r)$, $C(\E) \equiv C(\E,\R)$ will denote the space of $\R$-valued continuous functions on $\E$ topologized by uniform convergence on compacts. $C_b(\E)\equiv C_b(\E,\R)$ will denote the subspace of bounded $\R$-valued continuous functions on $\E$. $C_b(E)^+ \subset C_b(E)$ will denote space of all functions $f\in C_b(\E)$ with $\inf_{x \in E}f(x) >0$. If $\E$ is compact, then  $C(\E)  =C_b(\E)$ and consequently, the suffix $b$ will be dropped. If $\mu$ is $\s$-finite measure on $\E$, then as standard, $L^p(\E,\mu)$ will denote the space of $\R$ (or  depending on the context $\R^d$)-valued functions on $\E$ with finite $p$-th moment.  If there is no confusion about the space $\E$ or the measure $\mu$, then the notation $L^p(\mu)$ or $L^p(\E)$ will also be interchangeably used. Similarly, $L^\Phi(\E,\mu)\equiv L^\Phi(\E)\equiv L^\Phi(\mu)$ will denote the Orlicz space corresponding to a  Young's function $\Phi$. $\SC{M}_F(\E)$ (resp. $\SC{P}(\E) $) will denote the space of finite (resp. probability) Borel measures on $E$ with the topology being given by the weak convergence. $M_{\SC{P}(\E)}[0,\infty)$ will denote the space of measurable $\SC{P}(\E)$-valued functions on $[0,\infty)$.  $D_\E[0,\infty)$ will denote the space of cadlag functions taking values in $\E$ on the time interval $[0,\infty)$. $\l_\infty$ will denote the Lebesgue measure on $[0,\infty)$. For $A\subset \E$, $A^\delta$ will denote {\em $\delta$-fattening} of $A$, that is, $A^\delta = \{x \in \E: \inf_{y \in A} r(x,y) <\delta\}$. $\l_\infty$ will denote the Lebesgue measure on $[0,\infty).$

\setcounter{equation}{0}
\renewcommand {\theequation}{\arabic{section}.\arabic{equation}}
\section{Infinite-dimensional semimartingales}
\label{chap:Hsharp}
The goal of this section is to give a brief introduction to $\H^\#$-semimartingale, as introduced in \cite{KP96_2}, and describe stochastic integrals with respect to them.  A few other popular notions of infinite-dimensional semimartingales include  {\it orthogonal martingale random measure} \cite{GS79}, {\em worthy martingale random measures} \cite{Walsh86}, {\em Banach space-valued semimartingales} \cite{MP80}, {\em nuclear space-valued semimartingales} \cite{Us82}. But as noted in \cite{KP96_2}, most of these separate classes of processes can be thought of as $\H^\#$-semimartingales for suitable $\H$.

\subsection{$\H^\#$-semimartingale}
Let $\H$ be a separable Banach space.
\begin{definition}
 An $\R$-valued stochastic process $Y$ indexed by $\H\times [0,\infty)$ is an {\bf $\H^\#$-semimartingale} with respect to the filtration $\left\{\SC{F}_t\right\}$ if
\begin{itemize}
 \item for each $h \in \H$, $Y(h,\cdot)$ is a cadlag $\left\{\SC{F}_t\right\}$-semimartingale, with $Y(h,0) = 0$;
\item for each $t>0$,  $h_1,\hdots,h_m \in \H$ and $a_1,\hdots, a_m \in \R$, we have 
$$Y(\sum_{i=1}^m a_ih_i,t) = \sum_{i=1}^m a_i Y( h_i,t) \ \ \T{a.s}.$$
\end{itemize}
\end{definition}

As in almost all integration theory, the first step is to define the stochastic integral in a canonical way for simple functions and then to extend it to a broader class of integrands.

Let $Z$ be an $\H$-valued cadlag process of the form 
\begin{equation}
\label{simp}
Z(t) = \sum_{k=1}^m \xi_k(t)h_k,
\end{equation}
where the $\xi_k$ are  $\left\{\SC{F}_t\right\}$ adapted real valued cadlag processes, and $h_1,\hdots,h_k \in \H$.\\
The stochastic integral $Z_-\cdot Y$ is defined as
$$Z_-\cdot Y(t) = \sum_{k=1}^m \int_0^t \xi_{k}(s-) d Y(h_k,s).$$
Note that the integral above is just a real valued process. It is necessary to impose more conditions on the $\H^\#$-semimartingale to broaden the class of integrands $Z$.

\np
Let $\SC{S}_t$ be the collection of all processes of the form (\ref{simp}) with $\sup_{s\leq t}\|Z(s)\| \leq 1$.  Define
\begin{equation}
\label{simp_col}
\mathcal{H}_t =\left\{\sup_{s\leq t}|Z_-\cdot Y(s)|:Z \in \SC{S}_t \right\}.
\end{equation}

\begin{definition}
\label{hstd}
An $\H^\#$-semimartingale $Y$ is {\bf standard} if for each $t>0$, $\SC{H}_t$ is stochastically bounded, that is for every $t>0$ and $\e>0$ there exists $k(t,\e)$ such that
$$P\left[\sup_{s\leq t}|Z_-\cdot Y(s)| \geq k(t,\e)\right] \leq \e,$$
for all $Z \in \SC{S}_t$ . 
\end{definition}

\subsection{Integration with respect to a standard $\H^\#$-semimartingale}
Let $X$ be an $ \left\{\SC{F}_t\right\}$-adapted $\H$-valued cadlag process. Approximating $X$ by simple functions of the form (\ref{simp}) is a crucial technique  that is used repeatedly  in Section \ref{sec:ldp}. The following lemma on partition of  unity (Lemma 3.1, \cite{KP96_2}) and subsequent steps needed for such constructions are briefly discussed below. For a topological space $S$, let $C_b(S)$ denote the space of  continuous and bounded real-valued  functions on $S$ with the sup norm.

\begin{lemma}
\label{pou}
Let $(S,d)$ be a complete, separable metric space and $\left\{\phi_k\right\}$ a countable dense subset of $S$. Then for each $\e>0$, there exists a sequence $\left\{\psi_k^\e\right\} \subset C_b(S)$ such that $supp\left\{\psi_k^\e\right\} \subset B(\phi_k,\e), 0 \leq \psi_k^\e \leq 1, |\psi_k^\e(x) - \psi_k^\e(y)| \leq \f{4}{\e} d(x,y)$, and for all $x \in S$, $\sum_{k=1}^\infty \psi_k^\e(x) = 1$ where only finitely many terms in the sum are non zero. In fact, the $\psi_k^\e$ can be chosen such that for each compact $K \subset S$, there exists $N_K < \infty$ for which $\sum_{k=1}^{N_K} \psi_k^\e(x) = 1, x \in K$.
\end{lemma}

\np
Now let $S = \H$, and let $\left\{\phi_k\right\}$ be a countable dense subset of $\H$. Fix $\e>0$ and let $\left\{\psi_k^\e\right\}$ be as in Lemma \ref{pou}. For $x \in D_{\H}[0,\infty)$, define
$$x^\e(t) = \sum_k \psi_k^\e(x(t)) \phi_k.$$
Note that since $x$ is cadlag, for each $T>0$, there exists $N_T < \infty$ such that
$$x^\e(t) = \sum_{k=1}^{N_T} \psi_k^\e(x(t)) \phi_k,  \ \ \ t \in\left[0,T\right]. $$
Further observe that
$$\|x(t) - x^\e(t)\|_\H \leq \sum_k \psi_k^\e(x(t)) \|x(t) - \phi_k\|_\H \leq \e.$$
Let $X$ be a cadlag, $\H$-valued, $\left\{\SC{F}_t\right\}$-adapted process and similarly define
\begin{align}
 \label{xep_def}
X^\e(t) = \sum_k \psi_k^\e(X(t)) \phi_k.
\end{align}
Then as observed, $\|X - X^\e\|_\H \leq \e$ and the stochastic integral of $X^{\e}_-\cdot Y(t)$ is defined naturally as
$$X^\e_-\cdot Y(t) = \sum_k \int_0^t \psi_k^\e(X(s-)) d Y(\phi_k,s).$$
The following theorem [Theorem 3.11, \cite{KP96_2}] proves the existence of the limit of $\left\{X^\e_-\cdot Y\right\}$, which we define as the stochastic integral $X_-\cdot Y$.
\begin{theorem}
\label{SIdef}
 Let $Y$ be a standard $\H^\#$-semimartingale, and let $X$ be an $\H$-valued cadlag adapted process. Define $X^\e$ as above. Then 
$$X_-\cdot Y \equiv \lim_{\e\rt0}X^\e_-\cdot Y$$
exists in the sense that for each $t>0$,
$$\lim_{\e\rt0}P\left[\sup_{s\leq t}|X^\e_-\cdot Y(s) - X_-\cdot Y(s)|>\eta\right] = 0,$$
for all $\eta >0$. $X_-\cdot Y$ is a cadlag process.
\end{theorem}

\begin{example}
\label{whitenoise}
{\rm
Let $(\E,r)$ be a complete, separable metric space and $\mu$  a sigma finite measure on $(\E,\SC{B}(\E))$. Denote the Lebesgue measure on $[0,\infty)$ by $\lambda_\infty$, and let $W$ be a space-time Gaussian white noise on $\E\times[0,\infty)$ based on $\mu\ot\lambda_\infty$, that is, $W$ is a Gaussian process indexed by $\SC{B}(\E)\times [0,\infty)$
with $E(W(A,t)) = 0$ and $E(W(A,t)W(B,s)) = \mu(A\cap B) \min\left\{t,s\right\}$. For $h\in L^2(\mu)$, define
$W(h,t) = \int_{U\times [0,t)} h(x) W(dx, ds).$
The above integration is defined (see \cite{Walsh86}), and it follows that $W$ is an $\H^\#$-semimartingale with $\H=L^2(\mu)$. It is also easy to check that $W$ is standard in the sense of Definition \ref{hstd}.

}
\end{example}

\begin{example}
\label{poirndmeas} 
{\rm
Let $\E,r$ and $\lambda_\infty$ be as before  $\nu$  a sigma finite measure on $(\E,\SC{B}(\E))$. Let $\xi$ be a Poisson random measure on $\E\times [0,\infty)$ with mean measure $\nu\ot \lambda_\infty$, that is 
for each $\Gamma \in \SC{B}(\E)\ot \SC{B}([0,\infty))$, $\xi(\Gamma)$ is a Poisson random variable with mean $\nu\ot \lambda_\infty(\Gamma)$, and for disjoint $\Gamma_1$ and $\Gamma_2$,
$\xi(\Gamma_1)$ and $\xi(\Gamma_1)$ are independent. For $A \in \SC{B}(U)$ with $\nu(A)<\infty$, define
$\tilde{\xi}(A,t) = \xi(A\times \left[0,t\right]) - t\nu(A)$.
For $h \in L^2(\mu)$, let $\tilde{\xi}(h,t) = \int_{U\times [0,t)} h(x) \tilde{\xi}(dx, ds)$ and for $h \in L^1(\mu)$, let $\xi(h,t) = \int_{U\times [0,t)} h(x)\xi(dx, ds)$.
Then $\tilde{\xi}$ is a standard $\H^\#$-martingale with $\H=L^2(\nu)$ and $\xi$ is a standard $\H^\#$-semimartingale with $\H =L^1(\nu)$. The above indexing spaces can be changed as long as the corresponding integrations are defined.

}
\end{example} 

\begin{remark}{\rm
In fact, it can be shown that most worthy martingale random measures (in the sense of Walsh \cite{Walsh86}) or more generally semimartingale random measures are standard $\H^\#$-semimartingales for appropriate choice of indexing space $\H$ (see \cite{KP96_2}).}
\end{remark}

\subsection {$(\LL, \hat{\H})^\#$-semimartingale and infinite-dimensional stochastic integrals}
\label{lhsharp}
In the previous part, observe that the stochastic integrals with respect to infinite-dimensional standard $\H^\#$-semimartingales are real-valued.  Function valued stochastic integrals are of interest in many areas of infinite-dimensional stochastic analysis, for example, stochastic partial differential equations. With that in mind, we want to study stochastic integrals taking values in some infinite-dimensional space. If $Y$ is a standard $\H^\#$-semimartingale, we could put 
$H(x,t) = X(\cdot -, x)\cdot Y(t)$
where for each $x$ in a Polish space $\E$, $X(\cdot,x)$ is a cadlag process with values in $\H$. While the above integral is defined,  the function properties of $H$ are not immediately clear. Hence, a careful approach is needed for constructing infinite-dimensional stochastic integrals. For the integral process $X_-\cdot Y$ to take values in a Banach space $\LL$, the concept of a standard  $(\LL, \hat{\H})^\#$-semimartingale was introduced in  \cite{KP96_2}  as a natural analogue of  the standard $\H^\#$-semimartingale. Below, we give a brief outline of that theory.

Let $(\E,r_\E)$ and $(\U,r_\U)$ be two complete, separable metric spaces. Let  $\LL, \H $ be separable Banach spaces of $\R$-valued functions on $\E$ and $\U$ respectively. Note that for function spaces, the product $fg, f \in \LL, g \in \H$ has the natural interpretation of point-wise product. Suppose that  $\left\{f_i\right\}$ and $\left\{g_j\right\}$  are such that the finite linear combinations of the $f_i$ are dense in $\LL$, and the  finite linear combinations of the $g_j$ are dense in $\H$.

\begin{definition}
\label{Hhat}
Let $\hat{\H}$ be the completion of the linear space $\left\{\sum_{i=1}^l\sum_{j=1}^ma_{ij}f_ig_j: f_i \in \left\{f_i\right\}, g_j \in \left\{g_j\right\}\right\}$ with respect to some norm $\|\cdot\|_{\hat{\H}}$.
\end{definition}

For example, if
$$\le\|\sum_{i=1}^l\sum_{j=1}^ma_{ij}f_ig_j\ri\|_{\hat{\H}} = \sup\left\{\le|\sum_{i=1}^l\sum_{j=1}^ma_{ij}\<\l,f_i\>\<\eta,g_j\>\ri|: \l \in \LL^*, \eta \in \H^*, \|\l\|_{\LL^* }\leq 1, \|\eta\|_{\H^*} \leq 1\right\},$$
then $\hat{\H}$ can be interpreted as a subspace of  the space of bounded operators, $L(\H^*,\LL)$.\\

Let $\zeta_k = \sum_{i,j}a_{kij}f_ig_j, k =1,2,\hdots$ be a dense sequence in $\hat{\H}$, where in each sum only finitely many $a_{kij}$ are nonzero. Then Lemma \ref{pou} gives the partition functions $\left\{\psi_k^\e\right\}$ corresponding to the dense set $\left\{\zeta_k\right\}$, and for $x \in \hat{\H}$ defining 
$$x^\e = \sum_k\psi_k^\e(x)\zeta_k$$
as before, we have 
$$\|x^\e - x\|_{\hat{\H}} \leq \e.$$ 
Notice that $x^\e$ can be written as
\begin{align}
\label{xep}
x^\e& = \sum_{i,j}c_{ij}^\e(x)f_ig_j,
\end{align}
where $c_{ij}^\e(x) = \sum_{k}\psi_k^\e(x)a_{kij},$ and only finitely many $c_{ij}^\e(x)$ are non-zero.

With the above approximation in mind, denote $\SC{S}_{\hat{\H}}$  as the space of all processes $X \in D_{\hat{\H}}[0,\infty)$ of the form
$$X(t)= \sum_{ij}\xi_{ij}(t)f_ig_j,$$
where $\xi_{ij}$ are $\R$-valued, cadlag, adapted processes and only fintely many $\xi_{ij}$ are non zero.
If $Y$ is an $\H^\#$-semimartingale, and $X \in \SC{S}_{\hat{\H}}$ is of the above form define
$$X_-\cdot Y(t) = \sum_i f_i \sum_j \int_0^t \xi_{ij}(s-)\ dY(g_j,s).$$
Notice that $X_-\cdot Y \in D_{\LL}[0,\infty).$

\begin{definition}\label{std2}
An $\H^\#$-semimartingale is a {\bf standard $(\LL, \hat{\H})^\#$-semimartingale} if
$$\SC{H}_t \equiv \left\{\sup_{s\leq t}\|X_-\cdot Y(s)\|_\LL: X \in \SC{S}_{\hat{\H}},\ \sup_{s\leq t}\|X(s)\|_{\hat{\H}} \leq 1\right\}$$
is stochastically bounded for each $t > 0$.
\end{definition}

As in Theorem \ref{SIdef}, under the standardness assumption, the definition of $X_-\cdot Y$ can be extended to all cadlag $\hat{\H}$-valued processes $X$ by approximating $X$ by $X^\e$, where
$$X^\e(t) =\sum_k\psi_k^\e(X(t))\zeta_k = \sum_{i,j}c_{ij}^\e(X(t))f_ig_j.$$ 

\begin{remark} {\rm The standardness condition in Definition \ref{std2} will follow if there exists a constant $C(t)$ such that
$$E\left[\|X_-\cdot Y(t)\|_\LL\right] \leq C(t)$$
for all $X \in \SC{S}_{\hat{\H}}$ satisfying $\sup_{s\leq t}\|X(s)\|_{\hat{\H}} \leq 1$.
}
\end{remark}

\begin{remark}{\rm 
If $\H$ and $\LL$ are general Banach spaces (rather than Banach spaces of functions), then $\hat{\H}$ could be taken as the completion of $\LL\ot\H$ with respect to some norm, for example the projective norm (see \cite{Rr02}).
}
\end{remark}

\setcounter{equation}{0}
\renewcommand {\theequation}{\arabic{section}.\arabic{equation}}
\section{Uniform Exponential Tightness}
\label{sec:UET}
 This section concerns itself with extending the notion UET, as introduced in  \cite{Gar08} for finite-dimensional processes, to $\H^\#$-semimartingales and $(\LL, \hat{\H})^\#$-semimartingales. 
Our goal is to prove that a large deviation principle holds for the sequence of stochastic integrals $\left\{X_{n-}\cdot Y_n\right\}$, when  $\left\{(X_n,Y_n)\right\}$ satisfies a large deviation principle and the the driving integrators $Y_n$ form a UET sequence of $\H^\#$-semimartingales or $(\LL, \hat{\H})^\#$-semimartingales.
Since, in general it is not clear in which space  the $Y_n$ take values, the large deviation principle of $\left\{(X_n,Y_n)\right\}$ has  to be defined carefully (see the next section).

\np
Let $\left\{\SC{F}^n_t\right\}$ be a sequence of right continuous filtrations. Let $\mathcal{S}^n_t$ denote the space of all $\H$-valued processes $Z$, such that $\sup_{s\leq t}\|Z(s)\| \leq 1$ and is of the form
\[Z(t) = \sum_{k=1}^m \xi_k(t) h_k,\]
where the $\xi_k$ are cadlag and $\left\{\mathcal{F}^n_t\right\}$-adapted $\R$ valued processes and $h_1,\hdots,h_m \in \H$.\\
\begin{definition}
\label{UET_def}
A sequence of $\left\{\SC{F}^n_t\right\}$-adapted, standard $\H^\#$-semimartingales  $\left\{Y_n\right\}$ is {\bf uniformly exponentially tight} if,  for every $a > 0$ and $t>0$, there exists a $k(t,a)$ such that 
\begin{align}
\label{UET0}
  \limsup_n\f{1}{n}\sup_{Z \in \mathcal{S}^n} \log P\left[\sup_{s\leq t}|Z_-\cdot Y_n(s)| > k(t,a)\right] \leq -a.
\end{align}
\end{definition}

\begin{example}
\label{ex_Gauss}
{\rm
Let $\E,r, \mu$ and $\lambda_\infty$ be as in Example \ref{whitenoise}. Let $W$ denote space-time Gaussian white noise  on $\E\times [0,\infty)$ with $\<W(A,\cdot),W(B,\cdot)\>_t = \mu(A\cap B)t$. Consider $W$ as an $\H^\#$-semimartingale, with $\H = L^2(\mu)$, by defining
\[W(h,t) = \int_{\E \times[0,t)}h(x)W(dx,ds), \ \ \ h\in L^2(\mu).\]
We will show that $\left\{W_n \equiv n^{-1/2}W\right\}$ satisfies the UET condition. Let $Z$ be an adapted cadlag $L^2(\mu)$-valued process. Observe that $Z_-\cdot W$ is a continuous martingale with quadratic variation given by
\[\left[Z_-\cdot W\right]_t = \int_0^t \|Z(\cdot,s)\|_2^2 \ ds.\]
Therefore we can write $Z_-\cdot W$ as a time changed Brownian motion, where the time change is given by the quadratic variation $\left[Z_-\cdot W\right]_t$. More specifically,
\[Z_-\cdot W(t) = B_{\left[Z_-\cdot W\right]_t},\]
where $B$ is a standard Brownian Motion.
Now for $\sup_{s\leq t}\|Z(\cdot,s)\|_2 \leq 1$, we have that $\left[Z_-\cdot W\right]_t \leq t$. Thus, for $a > 0$
\begin{align*}
P(\sup_{s\leq t}|Z_-.W_n(s)| > K )& = P(\sup_{s\leq t}|Z_-.W(s)| > \sqrt{n}K ) =P(\sup_{s\leq t}|B_{\left[Z_-\cdot W\right]_s}| > \sqrt{n}K )\\
& \leq P(\sup_{s\leq t}|B_s| > \sqrt{n}K ) \ \ \ (\mbox{as } \left[Z_-\cdot W\right]_t \leq t) \\
& \leq 4\exp(-nK^2/2t).
\end{align*} 
Choosing $k(t,a) \equiv K = (2at)^{1/2}$, it follows that
\[\limsup_n\f{1}{n} \sup_{Z \in \mathcal{S}^n}\log P\left[\sup_{s\leq t}|Z_-\cdot W_n(s)| > k(t,a)\right] \leq -a,\]
and the UET condition is verified. }
\end{example}

\np
For the next example we take the indexing Banach space to be an appropriate {\it Orlicz space}: (see Section \ref{Orlicz}) 
\begin{example}
\label{ex_Poisson}
{\rm
Let $\xi_n$ be a Poisson random measure on $\E\times[0,\infty)$ with mean measure $n \nu \ot \l_\infty$, where $\nu$ is a $\s$-finite measure on $\E$. Then we show $\left\{Y_n \equiv \xi_n/n\right\}$ satisfies the UET condition when considered as a process indexed by the Banach space $\H = L^\Phi(\E,\nu) \equiv L^\Phi(\nu)$. Here $L^\Phi(\E,\nu)$ is the {\it Orlicz space} for the function $\Phi(x) = e^x - 1$ (see Section \ref{Orlicz}).\\
Let $Z$ be a $L^\Phi(\nu)$-valued cadlag process such that $\sup_{s\leq t}\|Z(\cdot,s)\|_\Phi \leq 1$.
Since
\[|Z_-\cdot \xi_n| \leq |Z_-|\cdot \xi_n, \]
without loss of generality we can take $Z \geq 0$ for our purpose.\\
We first estimate
$E(e^{Z_-\cdot \xi_n})$ by Ito's formula. Let $X_n(t) \equiv  Z_-\cdot \xi_n (t)$. For a $C^2$ function $f$, It{\^o}'s formula implies 
\begin{align*}
f(X_n(t))& =  f(X_n(0)) + \int_{\E\times \left[0,t\right]} f(X_n(s-) + Z(u,s-)) - f(X_n(s-))  \ \xi_n(du, ds).
\end{align*}
Taking $f(x) = e^x$, we get
\begin{align*}
E(e^{X_n(t)})& = 1 + nE\int_{\E\times [0,t)} e^{X_n(s)}(e^{Z(u,s)} - 1)\nu(du) \ ds\\
& = 1 + nE\int_0^t e^{X_n(s)}\int_\E(e^{Z(u,s)} - 1)\nu(du) \ ds.
\end{align*} 
Now since $\|f\|_\Phi \leq 1 $ iff $\int\Phi(|f|)\ d\nu \leq 1$, we see from our assumption on the process $Z$ that
$\sup_{s\leq t}\int_\E(e^{Z(u,s)} - 1) d\nu(u) \leq 1$. Thus,
\[E(e^{X_n(t)})\leq 1 + nE\int_0^t e^{X_n(s)}\ ds,\]
and by Gronwall's inequality
\[E(e^{X_n(t)}) = E(e^{Z_-\cdot \xi_n(t)}) \leq e^{nt}.\]
Therefore
\begin{align*}
P(\sup_{s\leq t} Z_-\cdot Y_n(s) > K )& =  P(Z_-\cdot \xi_n(t) > nK ) =P(e^{Z_-\cdot \xi_n(t)} > e^{nK} )\\
& \leq E(e^{Z_-\cdot \xi_n(t)}) /e^{nK} \leq e^{nt - nK}.
\end{align*}
Choosing $k(t,a) \equiv K= t+a$, we have
\[\limsup_n\f{1}{n} \sup_{Z \in \mathcal{S}^n}\log P\left[\sup_{s\leq t}|Z_-\cdot Y_n(s)| > k(t,a)\right] \leq -a.\]
}
\end{example}
\vspace{.2in}

The definition of uniform exponential tightness for a sequence of $(\LL, \hat{\H})^\#$-semimartingale is analogous to that of $\H^\#$-semimartingale with  the obvious change.\\
Let $\left\{\SC{F}^n_t\right\}$ be a sequence of right continuous filtrations. Recall that the Banach space $\hat{\H}$ was defined in Definition \ref{Hhat}.
Let $\mathcal{S}^n$ denote the collection of all $\hat{\H}$-valued processes $Z$, such that $\|Z(t)\|_{\hat{\H}} \leq 1$ and is of the form
\[Z(t) = \sum_{i,j=1}^{l,m} \xi_{i,j}(t) f_ig_j,\]
where the $\xi_{i,j}$ are cadlag and $\left\{\mathcal{F}^n_t\right\}$ adapted $\R$ valued processes, $\left\{f_i\right\}\subset \LL, \left\{g_j\right\}\subset \H$.

\begin{definition}
A sequence of $\left\{\SC{F}^n_t\right\}$-adapted $(\LL, \hat{\H})^\#$-semimartingales  $\left\{Y_n\right\}$ is {\bf uniformly exponentially tight} (UET) if,  for every $a > 0$ and $t>0$,     there exists a $k(t,a)$ such that 
\begin{align}
 \label{UET2}
 \limsup_n\f{1}{n}\sup_{Z \in \SC{S}^n} \log P\left[\sup_{s\leq t}\|Z_-\cdot Y_n(s)\|_\LL > k(t,a)\right] \leq -a.
\end{align}
\end{definition}

\setcounter{equation}{0}
\renewcommand {\theequation}{\arabic{section}.\arabic{equation}}
\section{Large deviations and exponential tightness}
\label{sec:ldp}
We define the exponential tightness and the large deviation principle for a sequence of $\H^\#$-semimartingales $\left\{Y_n\right\}$. Let $\H$ and $\K$ be two Banach spaces.
\begin{definition}
\label{exptight_def}
Let $\left\{Y_n\right\}$ be a sequence of $\left\{\SC{F}^n_t\right\}$-adapted $\H^\#$-semimartingales and $\left\{X_n\right\}$ be a sequence of cadlag, $\left\{\SC{F}^n_t\right\}$-adapted $\K$-valued processes. $\left\{(X_n,Y_n)\right\}$ is said to be {\bf exponentially tight}  if for every finite collection of elements  $\phi_1,\phi_2,\hdots \phi_k \in \H$, $\left\{(X_n,Y_n(\phi_1,\cdot),Y_n(\phi_2,\cdot),\hdots,Y_n(\phi_k,\cdot))\right\}$ is exponentially tight in
$D_{\K\times \R^k}[0,\infty)$.
\end{definition}

\begin{definition}
\label{LDP_def}
Let $\left\{Y_n\right\}$ be a sequence of $\left\{\SC{F}^n_t\right\}$-adapted $\H^\#$-semimartingales and $\left\{X_n\right\}$ be a sequence of cadlag, $\left\{\SC{F}^n_t\right\}$-adapted $\K$-valued processes. Let $A$ denote the index set consisting of all ordered finite subsets of $\H$. $\left\{(X_n,Y_n)\right\}$ is said to satisfy the {\bf large deviation principle}  with the rate function family $\left\{I_\alpha:\alpha \in A\right\}$ if for $\alpha =(\phi_1,\hdots,\phi_k),  \left\{(X_n,Y_n(\phi_1,\cdot),\hdots, Y_n(\phi_k,\cdot))\right\}$ satisfies a LDP in $D_{\K\times \R^k}[0,\infty)$ with the rate function $I_\alpha$.
\end{definition}

\np
The following  lemma shows the canonical consistencies that we expect among the family $\left\{I_\alpha:\alpha \in A\right\}$.

\begin{lemma} Let $X_n$ and $Y_n$ be as in Definition \ref{LDP_def} and suppose that   $\left\{(X_n,Y_n)\right\}$ satisfies a LDP with the rate function family $\left\{I_\alpha:\alpha \in A\right\}$.  Then the following assertions hold:
\begin{enumerate}[label={\rm (\roman*)}, leftmargin=*, align=right]
\item if $\alpha =(\phi_1,\phi_2,\hdots,\phi_k)$ and $\beta = (\phi_{i_1},\phi_{i_2},\hdots,\phi_{i_k})$ is a permutation of $\alpha$, then
\[I_\alpha(x,y_1, y_2, \hdots, y_k) =  I_\beta(x,y_{i_1}, y_{i_2}, \hdots, y_{i_k}), \ \ (x,y_1, y_2, \hdots, y_k) \in D_{\H\times \R^k}[0,\infty); \]

\item if $\alpha =(\phi_1,\phi_2,\hdots,\phi_k)$ and $\beta = (\phi_1,\phi_2,\hdots,\phi_k, \phi_{k+1})$, then
\begin{align*}
I_\alpha(x,y_1, y_2, \hdots, y_k)& = \inf_{y_{k+1}}\{I_\beta(x,y_1, y_2, \hdots, y_k,y_{k+1}): (x,y_1, y_2, \hdots, y_k,y_{k+1})\\ 
& \hspace{1.5cm}\in  D_{\H\times \R^{k+1}}[0,\infty)\}.
\end{align*}
\end{enumerate} 
\end{lemma}   

\np
\begin{proof} Notice that the permutation mapping $p: D_{\R^k}[0,\infty) \Rt D_{\R^k}[0,\infty)$, and the projection mapping
$\pi:D_{\R^{k+1}}[0,\infty) \Rt D_{\R^k}[0,\infty)$  defined respectively by 
\[p(y_1,y_2,\hdots, y_k) =  (y_{i_1}, y_{i_2}, \hdots, y_{i_k}) , \ \ \pi(y_1, y_2, \hdots, y_k,y_{k+1}) = (y_1, y_2, \hdots, y_k),\]
are continuous, and the theorem follows from the contraction principle.
\end{proof}

\begin{example}
\label{Gaussldp}
{\rm
Let $W$ be the space-time Gaussian white noise on  $(\E\times[0,\infty),\mu\ot\lambda_\infty)$ as in Example \ref{ex_Gauss}.
We saw earlier that $W$ forms a standard $\H^\#$-semimartingale with $\H = L^2(\mu)$.
We show below that  $\left\{W_n \equiv n^{-1/2}W\right\}$  satisfies the LDP in the sense of Definition \ref{LDP_def}.\\
First note that $W$ is a Gaussian process with stationary and independent increments, and  with covariance function
\[E(W(h_1,t) W(h_2,s)) = \<h_1,h_2\>(t\wedge s).\]
Here $\<\cdot,\cdot\>$ denotes the inner product in $L^2(\mu)$.\\
For a finite collection $\left\{h_1,\hdots,h_m\right\}$, $(W(h_1,\cdot),W(h_2,\cdot),\hdots,W(h_m,\cdot))$ is a Gaussian process with stationary and independent increments  with variance covariance matrix $t\Sigma_h$, where
\[ \Sigma_{h}= \left[\<h_i,h_j\>\right]_{i,j=1}^m, \ \ \ h = (h_1,\hdots,h_m).\]
Since $\Sigma_{h}$ is symmetric and non-negative definite,
\[\Sigma_{h} = C_{h}C_{h}^T.\]
It follows that
\[(W(h_1,\cdot),W(h_2,\cdot),\hdots,W(h_m,\cdot))  = C_{h}B(\cdot),\]
where $B$ is a standard $m$- dimensional Brownian Motion.
Now an application of the contraction principle and  Schilder's theorem implies that $\left\{(W_n(h_1,\cdot),W_n(h_2,\cdot),\hdots,W_n(h_m,\cdot))\right\}$ follows LDP 
with rate function $I_{h}(\psi)$, where
\[I_{h}(\psi) = \inf \left\{1/2 \int_0^\infty \|\dot{\phi}(t)\|^2 \ dt: \psi(\cdot) = C_h\phi(\cdot), \ \phi(t) = \int_0^t \dot{\phi}(u) \ du \mbox{ for some } \dot{\phi} \in L^2 \right\}.\]

}
\end{example}
\vspace{.7cm}

\begin{example}
\label{poildp}
{\rm
Let $\xi$ be a Poisson random measure on $\E\times[0,\infty)$ with mean measure $\nu \ot \l_\infty$. Define $Y_{n}(A,t) =\xi(A,[0,nt])/n$.  Then $\left\{Y_n\right\}$ satisfies LDP in the above sense, when the indexing Banach space $\H$ is taken to be  Morse-Transue space $M^{\Phi}(\nu) \subset L^{\Phi}(\nu)$ (see (\ref{MTspace}) in Section \ref{Orlicz}), for $\Phi(x) =e^x - 1$.
Note that for any finite collection $h = \left(h_1,\hdots,h_m\right)$, $(\xi(h_1,\cdot),\xi(h_2,\cdot),\hdots,\xi(h_m,\cdot))$ is a cadlag process with stationary and independent increments, hence an $m$-dimensional Levy process. Now Theorem 1.2 from de Acosta \cite{Acos94}  will establish a large deviation principle for  \\ $\left\{(Y_n(h,\cdot)\equiv Y_n(h_1,\cdot), Y_n(h_2,\cdot), \hdots, Y_n(h_m,\cdot))\right\}$, provided we verify the hypothesis:

$$E(\exp(\beta\sum_{i=1}^m|\xi(h_i,1)|))<\infty \ \mbox{ for every } \beta >0.$$ 
It is enough to show that 
$$E(\exp(\beta\sum_{i=1}^m\xi(|h_i|,1)))<\infty.$$
Note that from It\^{o}'s lemma, if $X(t) = \int_{U\times[0,t)} Z(u,s)\xi(du\times ds)$, then
\begin{align}
\label{ito1}
E\left[\exp(X(t))\right]& = 1 + E\int_{\E\times[0,t)} e^{X(s)}(e^{Z(u,s)}-1)\ \nu(du)  ds.
\end{align}
Choosing $Z(u,t) \equiv \beta\sum_{i=1}^m |h_i|(u)$, we get
$$E\left[\exp(\beta\sum_{i=1}^m\xi(|h_i|,1))\right] = \exp\le(\int_\E (e^{\beta\sum |h_i|(u)} -1) \ \nu(du)\ri)  < \infty.$$
The last inequality holds as $h_1,\hdots h_m\in M^\Phi(\nu)$ implies that $\beta\sum_{i=1}^mh_i \in M^\Phi(\nu)$.\\
The associated rate function of $\{Y_n(h,\cdot)\}$ in $D_{\R^m}[0,\infty)$ is given by
\begin{eqnarray}
 I_h(y) = 
\begin{cases}
\int_0^\infty \lambda_h(\dot{y}(s))\ ds, & \mbox{if } y_i(t) = \int_0^t \dot{y}_i(u) \ du \mbox{ for some } \dot{y}_i \in L^1[0,\infty), \\
\infty, & \mbox{otherwise,}
\end{cases}
\end{eqnarray}
where $\lambda_h$ is the Fenchel-Legendre transformation of 
$$\mu_h(x) = \log E(\exp(\sum_{i=1}^m x_i\xi(h_i,1))), \ \ \ x = (x_1,\hdots,x_m),$$
that is,
$$\lambda_h( y) = \sup_{x \in \R^m}\left[x\cdot y - \mu_h(x)\right].$$
Putting $Z(u,t) \equiv \sum_{i=1}^m x_i h_i(u)$ in (\ref{ito1}),
$$E\left[\exp(\sum_{i=1}^m x_i\xi(h_i,1))\right] = \exp\le[\int_\E \le(e^{\sum x_ih_i(u)} -1\ri) \ \nu(du)\ri].$$
It follows that
$$\mu_h(x) = \int_\E \le(e^{\sum x_ih_i(u)} -1\ri) \ \nu(du).$$
For $D_{\R^m}[0,T]$, the rate function $I_h$ admits the following alternate representation:
\begin{align}\label{rate_poi_alt}
I_h(y) = \inf\le\{L_T(\varphi): y_i(t) = \int_{\E\times [0,T]} h_i(u)\varphi(u,s)\nu(du)ds, \ i=1,2,\hdots,m\ri\}
\end{align}
where
\begin{align}\label{LT_poi}
L_T(\varphi) \equiv \int_{\E\times [0,T]} l(\varphi(u,s)) \nu(du) ds
\end{align}
with $l(z) = z\ln z - z +1.$

}
\end{example}

\subsection{LDP results for stochastic integrals: $\H^\#$-semimartingales}
The main aim of this section is to establish that exponential tightness and LDP hold for the integral $\{X_{n-}\cdot Y_n\}$. The path towards that starts by first studying finite-dimensional approximations $X^\e_{n-}\cdot Y_n$,
where recall that for a process $X$, $X^\e$ is defined by (\ref{xep_def}).
\begin{lemma}
\label{exptight1}
Let $\left\{Y_n\right\}$ be a sequence of $\left\{\SC{F}^n_t\right\}$-adapted, standard $\H^\#$-semimartingales and $\left\{X_n\right\}$  a sequence of cadlag, adapted $\H$-valued processes. Assume that $\left\{Y_n\right\}$ is UET. If $\left\{(X_n,Y_n)\right\}$ is exponentially tight in the sense of Definition \ref{exptight_def}, then $\left\{(X_n, Y_n, X^\epsilon_{n-}\cdot Y_n)\right\}$ is exponentially tight. 
\end{lemma}
\np
\begin{proof}  Let $\beta = (h_1,\hdots,h_m)$ be an ordered subset of $\H$. Denote
\[Y_n(\beta,\cdot) = (Y_n(h_1,\cdot)\hdots,Y_n(h_m,\cdot)). \]
We have to prove that $\left\{(X_n,Y_n(\beta,\cdot),X^\epsilon_{n-}\cdot Y_n)\right\}$ is exponentially tight in $D_{\H\times\R^m\times\R}$.

Since $\left\{X_n\right\}$ is exponentially tight, it satisfies the exponential compact containment condition (see Definition \ref{def:expcpt} and the paragraph below). Fix $a > 0$. Then there exists  a compact set $K_a \subset \H$ such that 
\[\limsup_n \f{1}{n} \log P\left[X_n(t) \notin K_a, \ \T{for some} \  t < a \right] \leq -a.\]
Let $\tau_{n,a} = \inf\left\{s: X_n(s) \notin K_a\right\}$. Then notice that
\[P\left[\tau_{n,a} < a\right] = P\left[X_n(t) \notin K_a, \ \T{for some} \  t < a \right]. \]
Hence 
\begin{align}
\limsup_n \f{1}{n} \log P\left[\tau_{n,a} < a\right] \leq -a.
\label{tna}
\end{align}

\np
For a stopping time $\tau$, define
$$X^{\tau-}(t) = X(t) 1_{\left[t<\tau\right]} + X(\tau-)1_{\left[t\geq\tau\right]}.$$
Notice that for each $t>0$, $X_n^{\tau_{n,a}-}(t) \in K_a$. Hence, there exists $N_a$ such that
\[X_n^{\epsilon, \tau_{n,a}-}(s) = \sum_{k=1}^{N_a} \psi_k^\epsilon (X_n^{\tau_{n,a}-}(s))\phi_k.\] 
Here $\left\{\psi^\epsilon_k\right\}$ is the partition of unity as in Lemma \ref{pou}.
Clearly, by the construction, 
\[\|X_n^{\epsilon, \tau_{n,a}-}(s) - X_n^{\tau_{n,a}-}(s)\|_\H \leq \epsilon.\]
Notice that for $t<\tau_{n,a}$
\begin{align*}
X_n^{\epsilon, \tau_{n,a}-} \cdot Y_n (t)& = \sum_{k=1}^{N_a}\int_0^t \psi_k^\epsilon (X_n^{\tau_{n,a}-}(s-)) \ dY_n(\phi_k, s)  \\
& = \sum_{k=1}^{N_a}\int_0^{t\wedge \tau_{n,a}} \psi_k^\epsilon (X_n(s-)) \ dY_n(\phi_k, s).
\end{align*}
Thus putting
\begin{align}
\label{zna}
Z_n^{a,\epsilon}(t)& = \sum_{k=1}^{N_a}\int_0^{t} \psi_k^\epsilon (X_n(s-)) \ dY_n(\phi_k, s),
\end{align}
we have 
\begin{align}
\label{eq1}
Z_n^{a,\epsilon}(t)& = X_{n-}^{\epsilon} \cdot Y_n (t) \ \ \T{for} \  t < \tau_{n,a}.
\end{align}
Since $\left\{Y_n\right\}$ is uniformly exponentially tight and $\left\{\psi_k^\epsilon (X_n(\cdot))\right\}$ is exponentially tight,  we deduce from Lemma 7.4 of Garcia \cite{Gar08} that $\left\{(X_n,Y_n(\beta,\cdot), Z_n^{a,\epsilon})\right\}$ is exponentially tight.

\np
Taking $\l(t) = t$ in the definition of metric $d$ in Ethier and Kurtz \cite[Page 117]{EK86}, which metrizes Skorohod topology on $D_\R[0,\infty)$, we get  
\begin{align}
\non
d(X^\epsilon_{n-}\cdot Y_n, Z_n^{a,\epsilon})& \leq \int_0^\infty  e^{-u}\sup_{t\geq 0} |Z_n^{a,\epsilon}(t\wedge u) - X^\epsilon_{n-}\cdot Y_n(t\wedge u)| \wedge 1 \ du\\ 
\non
& = \int_0^{\tau_{n,a}}(\hdots) + \int_{\tau_{n,a}}^\infty (\hdots) \\
& \leq e^{-\tau_{n,a}},
\label{xnyn1}
\end{align}
as the first integral in the second line in the above display is $0$  by (\ref{eq1}).
The same technique in fact gives us
\begin{align}
\label{xnyn2}
d((X_n,Y_n(\beta,\cdot),X^\epsilon_{n-}\cdot Y_n), (X_n,Y_n(\beta,\cdot),Z_n^{a,\epsilon}))& \leq e^{-\tau_{n,a}}.
\end{align}
Let $0<\delta<1$. Choose $a > -\log \delta > 0$ and notice that
\begin{align}
\non
\limsup_n \f{1}{n} \log P\left[d((X_n,Y_n(\beta,\cdot), X^\epsilon_{n-}\cdot Y_n), (X_n,Y_n(\beta,\cdot), Z_n^{a,\epsilon})) > \delta\right] \\
 \leq \limsup_n \f{1}{n} \log P (\tau_{n,a} < -\log \delta ) \leq -a.
\label{expapprox1}
\end{align}
Since $\left\{(X_n,Y_n(\beta,\cdot), Z_n^{a,\epsilon})\right\}$ is exponentially tight, there exists a compact set $F_a \subset D_{\H\times\R^m\times\R}$ such that
\[\limsup_n \f{1}{n} \log P\left[(X_n,Y_n(\beta,\cdot),Z_n^{a,\epsilon}) \notin F_a\right] \leq -a.\]
Recall that $F_a^\delta$ denotes $\delta$-fattening of the set $F_a$ (see Notations). As
\begin{align*}
\left\{(X_n,Y_n(\beta,\cdot),Z_n^{a,\epsilon}) \in F_a\right\}\cap \left\{d((X_n,Y_n(\beta,\cdot),X^\epsilon_{n-}\cdot Y_n), (X_n,Y_n(\beta,\cdot),Z_n^{a,\epsilon})) <\delta\right\}\\
\subset \left\{(X_n,Y_n(\beta,\cdot),X^\epsilon_{n-}\cdot Y_n) \in F_a^\delta\right\},
\end{align*}
we have 
\begin{align*}
\limsup_n \f{1}{n} \log P\left[(X_n,Y_n(\beta,\cdot),X^\epsilon_{n-}\cdot Y_n) \notin F_a^\delta\right]  \leq& \limsup_n \f{1}{n} \log P\left[((X_n,Y_n(\beta,\cdot),Z_n^{a,\epsilon}) \notin F_a\right] \\
  & \vee \log P \left[d((X_n,Y_n(\beta,\cdot),X^\epsilon_{n-}\cdot Y_n), (X_n,Y_n(\beta,\cdot),Z_n^{a,\epsilon})) > \delta\right] \\
 \leq & -a.
\end{align*}
Exponential tightness of $\left\{(X_n, Y_n(\beta,\cdot), X^\epsilon_{n-}\cdot Y_n)\right\}$ now follows from \cite[Lemma 3.3]{FK06}. 
\end{proof}

\begin{theorem}
\label{exptight2}
Let $\left\{Y_n\right\}$ be a sequence of $\left\{\SC{F}^n_t\right\}$-adapted, standard $\H^\#$-semimartingales and $\left\{X_n\right\}$  a sequence of cadlag, adapted $\H$-valued processes. Assume that  $\left\{Y_n\right\}$ is UET.  If $\left\{(X_n,Y_n)\right\}$ is exponentially tight in the sense of Definition \ref{exptight_def}, then $\left\{(X_n, Y_n, X_{n-}\cdot Y_n)\right\}$ is exponentially tight.
\end{theorem}

\np
\begin{proof} Let $\beta = (h_1,\hdots,h_m)$ be an ordered subset of $\H$. Define
$$Y_n(\beta,\cdot) = (Y_n(h_1,\cdot),\hdots,Y_n(h_m,\cdot)) . $$
We have to prove that $\left\{(X_n,Y_n(\beta,\cdot),X_{n-}\cdot Y_n)\right\}$ is exponentially tight in $D_{\H\times\R^m\times\R}$.\\
Since $\left\{(X_n,Y_n(\beta,\cdot),X^\epsilon_{n-}\cdot Y_n)\right\}$ is exponentially tight, for every $a > 0$, there exists a compact set $K_{a,\epsilon}$ such that
\begin{align}\label{exp-cpt-fin-int}
\limsup_n \f{1}{n} \log P\left[(X_n,Y_n(\beta,\cdot),X^\epsilon_{n-}\cdot Y_n) \notin K_{a,\epsilon}\right] \leq -a.
\end{align}
Since $\left\{Y_n\right\}$ is UET,  for every $a>0$, there exists $k(t,a)$ such that
\begin{align}\label{eq:Y-uet}
\limsup_n \f{1}{n}\log P\left[\sup_{s\leq t} |Z_n\cdot Y_n (s)| > k(t,a)\right] \leq -a,
\end{align}
for any sequence of cadlag $\left\{\SC{F}^n_t\right\}$-adapted $\left\{Z_n\right\}$  satisfying $\sup_{s\leq t}\|Z_n(s)\| \leq 1$. 
Without loss of generality, assume that $k(t,a)$ is nondecreasing right continuous function of $t$. \\
Now recall that $\|X_n^\epsilon(s) - X_n(s)\| \leq \epsilon$. Therefore taking $Z_n\equiv X_n^\epsilon - X_n$ in \eqref{eq:Y-uet}, we have
\begin{align}\label{xnyn3}
\limsup_n \f{1}{n}\log P\left[\sup_{s\leq t} |(X_n - X_n^\epsilon)\cdot Y_n (s)| > \epsilon \ k(t,a)\right] \leq -a.
\end{align}

\np
As before, taking $\lambda(s) = s$ in the definition of metric $d$ (see Ethier and Kurtz \cite[Page 117]{EK86}),  we have
\begin{align*}
d(X^\epsilon_{n-}\cdot Y_n, X_{n-}\cdot Y_n)& \leq \sup_{s\leq t} |(X_n- X_n^\epsilon)\cdot Y_n (s)| + e^{-t}\\
d((X_n,Y_n(\beta,\cdot),X^\epsilon_{n-}\cdot Y_n), (X_n,Y_n(\beta,\cdot), X_{n-}\cdot Y_n))& \leq \sup_{s\leq t} |(X_n - X_n^\epsilon)\cdot Y_n (s)| + e^{-t}.
\end{align*}

\np
Let $\delta > 0$. Notice that if we take $t>0$, such that $e^{-t}< \delta/2$, then
\begin{align}
\left\{d((X_n,Y_n(\beta,\cdot),X^\epsilon_{n-}\cdot Y_n), (X_n,Y_n(\beta,\cdot), X_{n-}\cdot Y_n)) > \delta\right\}& \subset \left\{\sup_{s\leq t} |(X_n - X_n^\epsilon)\cdot Y_n (s)| > \delta/2\right\}.
\label{xnyn4}
\end{align}
 Choose $t >0$ such that $e^{-t}< \delta/2$.  Then taking $\epsilon \equiv \epsilon_a $ such that $\epsilon_a \ k(t,a) \leq \delta/2 $, we get from \eqref{exp-cpt-fin-int}, \eqref{xnyn3} and \eqref{xnyn4} that
\begin{align*}
\limsup_n \f{1}{n}\log P [(X_n,Y_n(\beta,\cdot), X_{n-}\cdot Y_n) \notin K_{a,\epsilon_a}^{\delta}] & \leq \limsup_n \f{1}{n} \log P[(X_n,Y_n(\beta,\cdot),X^{\epsilon_a}_{n-}\cdot Y_n)\\
& \hspace{.4cm} \notin K_{a,\epsilon_a}] \vee \log P [d((X_n,Y_n(\beta,\cdot),X^{\epsilon_a}_{n-}\cdot Y_n), \\
& \hspace{.4cm} (X_n,Y_n(\beta,\cdot), X_{n-}\cdot Y_n)) > \delta] \ \ \leq -a.
\end{align*}
As before, from \cite[Lemma 3.3]{FK06} it follows that $\left\{(X_n,Y_n(\beta,\cdot), X_{n-}\cdot Y_n)\right\}$ is exponentially tight. 
\end{proof}

\np
We next prove the main theorem of this section.
\begin{theorem}
\label{LDP_mainth}
Let $\left\{Y_n\right\}$ be a sequence of $\left\{\SC{F}^n_t\right\}$-adapted, standard $\H^\#$-semimartingales and $\left\{X_n\right\}$ a sequence of cadlag, adapted $\H$-valued processes. Assume that $\left\{Y_n\right\}$ is UET. If $\left\{(X_n,Y_n)\right\}$ satisfies a LDP in the sense of Definition \ref{LDP_def}, with the rate function family $\left\{I_\alpha:\alpha \in A\right\}$ , then $\left\{(X_n, Y_n, X_{n-}\cdot Y_n)\right\}$ also satisfies a LDP.
\end{theorem} 

\begin{proof} The proof uses the same technique as in the proof of Lemma \ref{exptight1} and Theorem \ref{exptight2}. The same notation is used here as well. Let $\beta = (h_1\hdots,h_m)$ be a finite ordered subset of $\H$. We have to prove that
$\left\{(X_n,Y_n(\beta,\cdot),X_{n-}\cdot Y_n)\right\}$ satisfies a LDP with some rate function $J_{\beta}(\cdot,\cdot,\cdot)$.
\np
Let $\delta >0$, and choose $a>-\log \delta$. Define $Z_n^{a,\e}$ by (\ref{zna}) and $\tau_{n,a}$ as in Lemma \ref{exptight1}. Then from
(\ref{expapprox1}) ,
\begin{align}
\non
\limsup_n \f{1}{n} \log P\left[d((X_n,Y_n(\beta,\cdot)X^\epsilon_{n-}\cdot Y_n), (X_n,Y_n(\beta,\cdot),Z_n^{a,\epsilon})) > \delta\right]\\
 \leq \limsup_n \f{1}{n}  \log P (\tau_{n,a} < -\log \delta ) 
\leq  -a.
\label{LDPeq1}
\end{align}
Choose $t >0$ so that $e^{-t}< \delta/2$, and  then take $\epsilon \equiv \epsilon_a$ such that $\epsilon_a \ k(t,a) \leq \delta/2 $.  Using (\ref{xnyn4}) and the fact $\left\{Y_n\right\}$ is UET, we have
\begin{align}
\label{LDPeq2}
\limsup_n \f{1}{n} \log P \left[d((X_n,Y_n(\beta,\cdot),X^{\epsilon_a}_{n-}\cdot Y_n), (X_n,Y_n(\beta,\cdot),X_{n-}\cdot Y_n)) > \delta\right]
\leq -a.
\end{align}
Combining (\ref{LDPeq1}) and (\ref{LDPeq2}), it follows that
\begin{align*}
\limsup_n \f{1}{n} \log P \left[d((X_n,Y_n(\beta,\cdot),Z_n^{a,\epsilon_a}), (X_n,Y_n(\beta,\cdot),X_{n-}\cdot Y_n)) > \delta\right]
\leq -a.
\end{align*}
Now it follows from the finite dimensional result of Garcia \cite{Gar08} (also see Theorem \ref{Gar_th} in the Introduction) that $\left\{(X_n,Y_n(\beta,\cdot),Z_n^{a,\epsilon_a})\right\}$ satisfies a large deviation principle. Since $\delta \rt 0$ and $a \rt \infty$ implies that $\epsilon \rt 0$,  Lemma 3.14 of Feng and Kurtz proves that $\left\{(X_n,Y_n(\beta,\cdot),X_{n-}\cdot Y_n)\right\}$ satisfies a LDP with the rate function
\begin{align}
\label{rate3}
J_\beta(x,y^\beta,z) = \lim_{\eta\rt 0} \liminf_{\epsilon \rt 0} \liminf_{a\rt \infty}J^{a,\epsilon}_\beta(B((x,y^\beta,z),\eta)),
\end{align}
where $J^{ a,\epsilon}_\beta$ is the rate function for $\left\{(X_n,Y_n(\beta,\cdot),Z_n^{a,\e})\right\}$, and is given by
\begin{align}
\non
J^{a,\e}_{\beta}(x,y^{\beta},z) =& \inf_{y^a} \Big\{I_{(\alpha_a, \beta)}(x,(y^a, y^\beta)): z = \sum_{k=1}^{N_a}\psi_k^\e(x)\cdot y^a_k,\ (x,y^a, y^\beta) \in D_{\H\times\R^{N_a+m}}[0,\infty), \\
& \qquad \ (y^a,y^\beta) \ \T{finite variation} \Big\}.  
\end{align}
\end{proof}

\setcounter{equation}{0}
\renewcommand {\theequation}{\arabic{section}.\arabic{equation}}
\section{Identification of the rate function}
\label{sec:rateid}
The important starting point here is to work with a suitable collection, $\{Y(\phi_k,\cdot)\}$ or equivalently $\{\phi_k\}\subset \H$, which is enough to determine the LDP of $\{X_{n-}\cdot Y_n\}$. For a large class of separable Banach spaces, this can be achieved by working with a Scahuder basis $\left\{(\phi_k, p_k)\right\}$ (see Definition \ref{Schau}). For instance, if $\H$ is separable Hilbert space, any complete orthonormal system ${\phi_k}$ forms a Schauder basis.
For those separable Banach spaces not having a Schauder basis, we work with a  pseudo-basis $\left\{(\phi_k, p_k)\right\}$, with $\phi_k \in \H, p_k \in C(\H,\R)$. In other words, for every $h \in \H$, we have
\begin{align}
h &=  \sum_{k=1}^\infty p_k(h)\phi_k.
\label{Schau1}
\end{align}
Throughout the rest of the paper, whenever $\H$ does not have a Schauder basis, we work with a pseudo-basis $\left\{(\phi_k, p_k)\right\}$ satisfying \ref{2} of Theorem \ref{hom}. As stated in Theorem \ref{hom}, this always exists for separable Banach spaces.

\vs*{0.2cm}
\np
{\bf Relevant notations:} The following  notations  will be used in the subsequent discussions. 

\begin{itemize}
 \item $\alpha_k = (\phi_1,\hdots,\phi_k)$.
\item  $\alpha = (\phi_1,\phi_2,\hdots)$.
\item $P_{N}(h) = (p_1(h),\hdots,p_N(h))$.
\item $ y = (y_1,y_2,\hdots) \in D_{\R^\infty}[0,\infty)$.
\item $ y^{(k)}  = (y_1,\hdots,y_k) \in D_{\R^k}[0,\infty)$.
\item If $Y_n$ is an $\H^\#$-semimartingale then $Y_n(\alpha_k,\cdot) = (Y_n(\phi_1,\cdot),\hdots, Y_n(\phi_k,\cdot))$ and \\
$Y_n(\alpha,\cdot) = (Y_n(\phi_1,\cdot), Y_n(\phi_2,\cdot),\hdots)$.
\item $I_{\alpha_k}$ and $I_{\alpha}$ will denote the rate function of $\left\{(X_n,Y_n(\alpha_k,\cdot))\right\}$ and $\left\{(X_n,Y_n(\alpha,\cdot))\right\}$ respectively.
Similarly,  $J_{\alpha_k}$ and $J_{\alpha}$ will denote the rate function of $\left\{(X_n,Y_n(\alpha_k,\cdot),X_{n-}\cdot Y_n)\right\}$ and $\left\{(X_n,Y_n(\alpha,\cdot), X_{n-}\cdot Y_n)\right\}$ respectively.
\end{itemize}

\np
Since $X_n \in D_{\H}[0,\infty)$, we have
$X_n(t) \equiv \sum_{k=1}^\infty p_k(X_n(t))\phi_k.$
Let
$$X^{(m)}_n(t) \equiv \sum_{k=1}^m p_k(X_n(t))\phi_k.$$
Theorem \ref{LDP_mainth} and the result of Garcia \cite[Theorem 1.2]{Gar08}) shows that LDP holds for the tuples \\ $(X_n,Y_n(\alpha_k,\cdot),X^{(m)}_{n-}\cdot Y_n)$ and $(X_n,Y_n(\alpha,\cdot), X^{(m)}_{n-}\cdot Y_n)$, and their rate functions, denoted respectively by $J^m_{\alpha_k}$ and $J^m_{\alpha}$, are as follows.
For $k\geq m$,
\begin{equation}
\begin{aligned}
J^m_{\alpha_k}(x, y^{(k)},z) = 
\begin{cases} 
I_{\alpha_k}(x, y^{(k)})&  \ \  \T{if} \ z(t) = \sum_{j=1}^m p_j(x)\cdot y_j(t), \ \ y^{(k)} \ \T{finite variation},\\
  \infty &  \ \  \T{otherwise.}
\end{cases}
\end{aligned}
\label{findim1}
\end{equation}
By the  Dawson-Gartner  theorem \cite{DZ98}, 
\begin{align}
\label{findim2}
J^m_{\alpha}(x, y,z)& = \sup_{k}J^m_{\alpha_k}(x, y^{(k)},z).
\end{align}
Finally, as in \eqref{rate3}, we have
$$J_{\alpha}(x, y,z) = \lim_{\eta \rt 0}\liminf_{m\rt\infty}J^m_{\alpha}(B_{\eta}((x, y,z))),$$
where $B_{\eta}((x, y,z))$ is the ball of radius $\eta$ in $D_{\H\times\R^{\infty}\times\R}$.\\

The UET property of the sequence $\{Y_n\}$ has several interesting consequences. First of all, we show that if $I_\alpha(x, y) < \infty$, then 
$\y(t)\equiv\sum_{j} y_j(t)p_j$ exists as an element of $\H^*$, where  $\left\{(\phi_{k},p_{k})\right\}$ is a pseudo-basis of $\H$ satisfying \ref{2} of Theorem \ref{hom}. When $\H$ is a Hilbert space then notice that the $p_j \in \H^* = \H$ are orthogonal and the fact that $\sum_{j} y_j(t)p_j \in \H$ can be proved by simply showing that $\sum_{j}|y_j(t)|^2 < \infty$. 
For a general Banach space $\H$, the  proof goes by first showing  the convergence of $\sum_{j} y_j(t)p_j $ in $C(\H,\R)$ topologized by the family of seminorm $\left\{\sigma_C: C\subset \H \mbox{ \ compact}\right\} $, where the $\sigma_C$ are defined by
$$\sigma_C(f) = \sup_{h\in C} |f(h)|.$$
Let $\H^*_c$ denotes the subspace $\H^*$ of $C(\H,\R)$ equipped with the corresponding subspace topology. This is equivalent to the topology of uniform convergence on compacts and is weaker than the norm topology on $\H^*$ which is equivalent to the topology of uniform convergence on closed balls.

The next result which follows due to the UET property is that for every $t>0$, the mapping 
$$h \in \H \Rt \sum_{j} y_j(t)p_j(h)\equiv \y(t) $$
is linear whenever $I_\alpha(x, y) < \infty$. Finally, we show that $\y$ has finite total variation over any finite interval in the following sense. For $\y \in D_{\H^*_c}[0,\infty)$, define the total variation of $\y$ in the interval $[0,t)$ as 
$$T_t (\y) = \sup_{\sigma} \sum_i \|\y(t_i) -\y(t_{i-1})\|_{\H^*}.$$
If $\y \in D_{\H^*_c}[0,\infty)$ is such that $T_t (\y) < \infty$, then the integral $x\cdot \y$ can be defined as
$$ x\cdot \y = \lim_{\|\s\| \rt 0 }\sum_{i} \<x(t_i), \y(t_{i+1}) -\y(t_i)\>_{\H,\H^*},$$
where $\s =\left\{t_i\right\}$ is a partition of $\left[0,t\right]$, and $\|\s\|$ denotes the mesh of the partition $\s$ (see  \ref{vectint} in the Appendix).\\

\begin{theorem}
\label{UETprop}
Let $\H$ be a separable Banach space. Let  $\left\{(\phi_{k},p_{k})\right\}$ be a pseudo-basis of $\H$ satisfying \ref{2} of Theorem \ref{hom}, or a Schauder basis, if the latter exists. Suppose that $\{Y_n\}$ is a UET sequence and that for $(x,y) \in D_{\H\times \R^\infty}[0,\infty)$,  $I_\alpha(x,y)<\infty$. Then,
\begin{enumerate}[label={\rm (\roman*)}, leftmargin=*, align=right]
\item \label{yinh2} for every compact set $C \subset \H$,
$$\sup_{t\leq T}\sup_{h\in C}\le|\sum_{j=M}^{N}y_j(t)p_j(h)\ri| \rt 0, \ \ \ \mbox{as} \ \ M, N \rt \infty;$$
\item \label{ylinear} for every $t>0$, $\y(t) \in \H^*$, where $\y(t)\equiv\sum_{j} y_j(t)p_j$;
\item \label{yfv} for every $t>0$, $T_t(\y) < \infty$.
\end{enumerate}
\end{theorem}

\begin{remark}
Notice that by \ref{yinh2} of Theorem \ref{UETprop}, if $I_{\alpha}(x,y) < \infty$,  then  for each $t>0$, $\y(t) \in C(\H,\R)$, where
\begin{align}
 \label{ydef2}
\y(t) = \sum_{j} y_j(t)p_j.
\end{align}
In fact, from the conclusion of the theorem it follows that $\y$ is a cadlag function in the time variable, that is $\y \in D_{C(\H,\R)}[0,\infty).$
Part \ref{ylinear} of Theorem \ref{UETprop} tells that  if $I_\alpha(x,y)<\infty$, then  $\y \in D_{\H^*_c}[0,\infty)$.
\end{remark}

\np
In the Hilbert space setting, one can actually prove the stronger statement. Indeed, identifying $\H$ with $\H^*$ in this case, $p_k(\cdot) =\<\cdot,\phi_k\>$ can be identified with $\phi_k$, where $\{\phi_k\}$ forms the complete orthonormal basis of $\H$, and consequently we prove that if $I_{\alpha}(x,y)<\infty$, then $y = \sum_k y_k\phi_k \in D_{\H}[0,\infty)$ (which has a stronger topology than $ D_{\H^*_c}[0,\infty)$). The following theorem establishes this stronger statement when $\H$ is a separable Hilbert space. The key techniques of the proof are very similar (in fact, simpler) to that of Theorem \ref{UETprop}, and therefore we only prove the result in the more general Banach space setting.

\begin{theorem}
\label{yinh}
Let $\H$ be a separable Hilbert space with an orthonormal basis $\{\phi_k\}$. Suppose that $\{Y_n\}$ is a UET sequence. Suppose that for $(x, y) \in D_{\H\times \R^\infty}[0,\infty)$, $I_\alpha(x, y) < \infty$. Then
\begin{enumerate}[label={\rm (\roman*)}, leftmargin=*, align=right]
\item \label{yh}$\displaystyle{\sup_{t\leq T}\sum_{j}|y_j(t)|^2 < \infty, \ \ \ \T{for all} \ T>0}$;
\item \label{yconv}$\displaystyle{\sup_{t\leq T}\sum_{j=M}^N|y_j(t)|^2 \rt 0, \ \ \ \T{as } M, N\rt \infty}$;
\item \label{fv} for every $t>0$, $T_t(\y) < \infty$.
\end{enumerate}
\end{theorem}

\begin{proof} ({\bf Theorem \ref{UETprop}}) The proofs go by the method of contradiction. Specifically, in each case assuming that the conclusion to be not true, we arrive at a contradiction by showing that $I_\alpha(x, y) = \infty$.

\ref{yinh2} Let $\left\{(\phi_{k},p_{k})\right\}$ be a pseudo-basis of $\H$ satisfying \ref{2} of Theorem \ref{hom}.
Fix an $a>0$. For $T>0$, define $k(T,a)$ by (\ref{UET0}). If the result is not true, then there exist a $\rho < (k(T,a)+1)^{-1}$ and a compact set $C$ such that for all $N_0$, there exist $ N > M > N_0$ and  a $0<t<T$ such that
$$\sup_{h\in C}|\sum_{j=M}^{N}y_j(t)p_j(h)| > 2\rho.$$
Since $\sup_{h\in C}\|\sum_{k=1}^N p_k(h)\phi_k - h\|_\H \rt 0$, there exists an $N_0$ such that for all $M,N > N_0$
\begin{align}
\label{bd}
\sup_{h\in C}\|\sum_{k=M}^N p_k(h)\phi_k\|_\H < \rho^2. 
\end{align}
For this $N_0$, find $N> M > N_0$ such that
$$\sup_{t\leq T}\sup_{h\in C}|\sum_{j=M}^{N}y_j(t)p_j(h)| > 2\rho.$$
Next find a  $0 <t < T$ (depending on $M,N$ and $T$) and a $\gamma_t \in C$ such that  
\begin{align}
\label{SN1}
|\sum_{j=M}^N y_j(t)p_j(\gamma_t)| > \rho.
\end{align}
Without loss of generality assume that $t$ is a continuity point of $y_i, i=M,\hdots,N$.

By the continuity of the projection $\pi_t:D_{\R^\infty}[0,\infty) \Rt \R$ at $ y$, for every $\e>0$, there exists an $r>0$, such that
\[d( u, y)<r \ \ \RT \ \ |u_i(t) - y_i(t)| < \e, \ \ \ i=M,\hdots, N.\]
Choose $\e = \rho^2(\sup_{h\in C}\sum_{M\leq k\leq N}|p_k(h)|)^{-1}$.

Define
$$z(s) = \rho^{-2}\sum_{k=M}^N p_k(\gamma_t)1_{[0,t)}(s)\phi_k.$$
Observe that by (\ref{bd}), $\sup_{s\leq t}\|z(s)\|_\H \leq 1$. We claim that\\

\np
{\bf Claim:} $\left\{d(Y_n(\alpha,\cdot), y) < r\right\} \subset \left\{z\cdot Y_n(t) > k(t,a)\right\}$\\
{\it Proof of the claim:} Let $\w \in LHS$.
Then 
\begin{align*}
 z\cdot Y_n(\w,t) & =  \rho^{-2}\sum_{k=M}^N p_k(\gamma_t)Y_n(\phi_k,t)(\w)\\
& = \rho^{-2} \sum_{k=M}^N p_k(\gamma_t)y_k(t)+ \rho^{-2}p_k(\gamma_t)(Y_n(\phi_k,t)(\w) - y_k(t)).\\
\end{align*}
It follows from (\ref{SN1}) that
\begin{align*}
|z\cdot Y_n(\w,t) | & \geq \rho^{-1} - \rho^{-2}\e \sum_{M\leq k\leq N} |p_k(\gamma_t)| \\
&  \geq k(T,a)+1 - \rho^{-2}\e \sup_{h\in C}\sum_{M\leq k\leq N} |p_k(h)|\\
& \geq k(T,a) \ \ \ \ (\mbox{ by the choice of } \e).
\end{align*}
Therefore we get
\begin{align*}
P\left[(X_n,Y_n(\alpha,\cdot)) \in B_{r}(x, y)\right] & \leq P\left[d(Y_n(\alpha,\cdot), y) <r\right]\\
& \leq P\left[z\cdot Y_n(t) > k(T,a)\right]\\
& \leq P\left[\sup_{s\leq T }|z_-\cdot Y_n(s)| > k(T,a)\right].\\
\end{align*}
Hence, using the UET condition (\ref{UET0}), we find
\begin{align}\label{cont_step}
-I_{\alpha}(x, y) \leq \limsup_{n\rt\infty} \f{1}{n} \log P\left[(X_n,Y_n(\alpha,\cdot)) \in B_r(x, y)\right] \leq -a.
\end{align}
Since this is true for all $a$, $I_{\alpha}(x,y) =\infty$ and  we are done.

\ref{ylinear}
We have to prove that 
\begin{enumerate}
\item $\y(t)(g+h) = \y(t)(g) + \y(t)h, \ \ g,h \in \H.$
\item $\y(t)(ch) = c\y(t)(h), \ \ c \in \R, h\in \H.$
\end{enumerate}
We prove the first claim and the proof of the second will be similar. Note that once we prove that $\y(t)$ is a linear functional, the fact that it is a continuous linear functional will follow from the previous part.\\

 If the conclusion is false, then there exist $g, h \in \H$ such that $\y(t)(g+h) \neq \y(t)(g) + \y(t)(h)$. Fix an $a>0$.  Then there exists a $\kappa < (k(t,a)+1)^{-1}$
\begin{align}
 \label{cont1}
|\y(t)(g+h) - \y(t)(g) - \y(t)(h)|> 2\kappa >0.
\end{align}
Since 
$$\sum_{j=1}^N \left[p_j(g+h)\phi_j - p_j(g)\phi_j-p_j(h)\phi_j\right] \rt 0, \ \ \mbox{as} \ \ N \rt \infty,$$
find an $N_0$ such that for all $N>N_0$ 
$$\|\sum_{j=1}^N \left[p_j(g+h)\phi_j - p_j(g)\phi_j-p_j(h)\phi_j\right]\|_{\H} < \kappa^2.$$
For this $N_0$ find an $N>N_{0}$ such that 
\begin{align}
 \label{cont2}
|\sum_{j=1}^N\left[y_j(t)p_j(g+h) - y_j(t)p_j(g) - y_j(t)p_j(h)\right]| > \kappa.
\end{align}
Without loss of generality assume that $t$ is a continuity point of $(y_1,\hdots,y_N)$.
By the continuity of the projection $\pi_t:D_{\R^N}[0,\infty) \Rt \R$ at $ y$, for every $\e>0$, there exists an $r>0$ such that
\[d( u, y)<r \ \ \RT \ \ |u_i(t) - y_i(t)| < \e, \ \ \ i=1,\hdots, N.\]
Choose $\e = \kappa^2(\sum_{1\leq k\leq N} |p_k(g+h) - p_k(g) -p_k(h)|)^{-1}$.

Define
$$z(s) = \kappa^{-2}\sum_{k=1}^N\left[ p_k(g+h) - p_k(g) -p_k(h)\right]1_{[0,t)}(s)\phi_k.$$
Observe that $\sup_{s\leq t}\|z(s)\|_\H \leq 1$. We claim that\\

\np
{\bf Claim:} $\left\{d(Y_n(\alpha,\cdot), y) < r\right\} \subset \left\{z\cdot Y_n(t) > k(t,a)\right\}$\\
{\it Proof of the claim:} Let $\w \in LHS$.
Then 
\begin{align*}
 z\cdot Y_n(\w,t) & =  \kappa^{-2}\sum_{k=1}^N \left[ p_k(g+h) - p_k(g) -p_k(h)\right]Y_n(\phi_k,t)(\w)\\
& = \kappa^{-2} \sum_{k=1}^N \left[ p_k(g+h) - p_k(g) -p_k(h)\right]y_k(t) \\
& \hspace{.7cm} +\kappa^{-2}\sum_{k=1}^N\left[ p_k(g+h) - p_k(g) -p_k(h)\right](Y_n(\phi_k,t)(\w) - y_k(t)).\\
\end{align*}
It follows from (\ref{cont2}) that
\begin{align*}
|z\cdot Y_n(\w,t) | & \geq \kappa^{-1} - \kappa^{-2}\e \sum_{1\leq k\leq N}|p_k(g+h) - p_k(g) -p_k(h)| \\
&   \geq k(t,a)+1 - \kappa^{-2}\e \sum_{1\leq k\leq N} |p_k(g+h) - p_k(g) -p_k(h)|\\
& \geq k(t,a) \ \ \ \ (\mbox{ by the choice of } \e)
\end{align*}
The rest of the proof is same as that of  \ref{yinh2}.

\ref{yfv}
Fix an $a>0$. If the assertion is not true, then we can find a partition $\left\{t_i\right\}_{i=1}^p$ such that
\[\sum_{i=1}^p\|\y(t_{i}) - \y(t_{i-1})\|_{\H^*} > (1+\kappa)(K+3),\]
where $\kappa>0$ and $K =k(t,a)$ is as defined in (\ref{UET0}) . 
Then for each $i=1,\hdots p$, find $\gamma_{i-1} \in \H$ with $\|\gamma_{i-1}\|_\H \leq 1$, such that
\[\sum_{i=1}^p\|\y(t_{i})(\gamma_{i-1}) - \y(t_{i-1})(\gamma_{i-1})\|_{\H^*} > (1+\kappa)(K+2).\]
From  the definition $\y$ there exists an $N>0$ such that
$$\sum_{i=1}^p|\sum_{j=1}^N(y_j(t_{i})-y_j(t_{i-1}))p_j(\gamma_{i-1})| > (1+\kappa)(K+1).$$
Since $\sum_j p_j(\gamma_{i-1})\phi_j = \gamma_{i-1}$, choose $N$ large enough so that
\begin{align}
\label{1}
\|\sum_{j=1}^N p_j(\gamma_{i-1})\phi_j\|_{\H} \leq (1+\kappa)\|\gamma_{i-1}\|_{\H} \leq (1+\kappa), \ \ i =1,\hdots,p.
\end{align}
Without loss of generality, assume that $ y =(y_1,y_2,\hdots)$ is continuous at $\left\{t_i\right\}_{i=1}^p$.
By the continuity of the projections $\pi_{t_i}:D_{\R^\infty}[0,\infty) \Rt \R^N$ at $ y$, for every $\e>0$, there exists an $r>0$, such that
\[d( u, y)<r \ \ \RT \ \ |u_k(t_i) - y_k(t_i)| < \e, \ \ \ k=1,\hdots, N, i=1,\hdots,p.\]
Choose $\e = (1+\kappa)(2p\sum_{k\leq N}B_k)^{-1}$, where $B_k= \max_i \{|p_k(\g_{i-1})|\}$. Define 
\[z(s) = (1+\kappa)^{-1}\sum_{i=1}^p\rho_{i-1}\sum_{j=1}^Np_j(\gamma_{i-1})1_{[t_{i-1},t_i)}(s)\phi_j,\]
where $\rho_{i-1} =\mbox{sgn}\left[\sum_{j=1}^N{p_j(\gamma_{i-1})(y_j(t_i) -y_j(t_{i-1}))}\right]$.\\
Notice that (\ref{1}) implies $\|z(s)\| \leq 1$. As before, we claim 
$$\left\{d(Y_n(\alpha,\cdot), y) < r\right\} \subset \left\{z\cdot Y_n(t) > K\right\}.$$
To see this notice if $\omega \in LHS$, then
\begin{align*}
z\cdot Y_n(t)(\omega)&=  (1+\kappa)^{-1}\sum_{i=1}^p\rho_{i-1}\sum_{j=1}^N p_j(\gamma_{i-1})(Y_n(\phi_j,t_i) - Y_n(\phi_j,t_{i-1}) )\\
& = (1+\kappa)^{-1}\sum_{i=1}^p\rho_{i-1}\sum_{j=1}^N p_j(\gamma_{i-1})(y_j(t_i) - y_j(t_{i-1}))\\
& \hspace{.7cm} + (1+\kappa)^{-1}\sum_{i=1}^p\rho_{i-1}\sum_{j=1}^N p_j(\gamma_{i-1})(Y_n(\phi_j,t_i)- y_j(t_{i})) \\
&\hspace{.7cm}- (1+\kappa)^{-1}\sum_{i=1}^p\rho_{i-1}\sum_{j=1}^N p_j(\gamma_{i-1})(Y_n(\phi_j,t_{i-1}) - y_j(t_{i-1})).\\
\end{align*}
Hence
\begin{align*}
|z\cdot Y_n(t)(\omega)|& \geq (1+\kappa)^{-1}\sum_{i=1}^p|\sum_{j=1}^N p_j(\gamma_{i-1})(y_j(t_i) - y_j(t_{i-1}))| \\
&\hspace{.7cm}- (1+\kappa)^{-1}\sum_{i=1}^p\rho_{i-1}\sum_{j=1}^N p_j(\gamma_{i-1})|Y_n(\phi_j,t_i)- y_j(t_{i})| \\
& \hspace{.7cm}- (1+\kappa)^{-1}\sum_{i=1}^p\rho_{i-1}\sum_{j=1}^N p_j(\gamma_{i-1})|Y_n(\phi_j,t_{i-1}) - y_j(t_{i-1})|\\
& \geq K+1 - 2(1+\kappa)^{-1}\epsilon p\sum_{j=1}^NB_j = K.
\end{align*}
Again, the rest of the proof is similar to that of \ref{yinh2}. 

\end{proof}

Let 
\begin{align}
 \label{DHdef}
\SC{D} = \left\{\y\in D_{\H^*_c}[0,\infty): T_t(\y) < \infty, \ \mbox{ for all } \ t>0\right\}.
\end{align}

\begin{remark}
Part \ref{yfv} of Theorem \ref{UETprop} indicates that if for  $(x, y) \in D_{\H\times \R^\infty}[0,\infty)$,  $I_\alpha(x, y) < \infty$, then $\y \in \SC{D}$, where $\y$ is defined by (\ref{ydef2}).
\end{remark}

\np
For  $g\in D_{\R^k}[0,\infty)$, define
\begin{align*}
g_\delta (t) = \sum_{k}g(\tau_k^\delta)1_{[\tau_k^\delta,\tau_{k+1}^\delta)}(t),
\end{align*}
where the $\tau_k^\delta$ are defined by:
\begin{align*}
\tau_0^\delta & = 0,\\
\tau_{k+1}^\delta & = \inf\left\{s>\tau_k^\delta: |g(s) - g(\tau_k^\delta)| >\delta\right\}.
\end{align*}
Clearly, $\sup_t |g_\delta(t) - g(t)| \leq \delta$. For $y \in D_{\R^k}[0,\infty) $, note that
\[g_\delta\cdot y(t) = \sum_k g(\tau_k^\delta)^T(y(\tau_{k+1}^\delta\wedge t) - y(\tau_k^\delta\wedge t)). \]

\vspace{.5cm}
\np
Recall that for $h\in \H$, the notation $P_N(h)$ was defined at the beginning of the Section \ref{sec:rateid}.
\begin{lemma}
\label{approx0}
For $\eta>0$ and $a>0$, there exist sufficiently small $\delta>0$ and sufficiently large  $N>0$ such that
\[\limsup_{n\rt\infty}\f{1}{n}\log P\left[d(X_{n-}\cdot Y_n, (P_N(X_{n-}))_\delta\cdot Y_n(\alpha_N,\cdot))>2\eta\right] \leq -a.\]
\end{lemma}

\begin{proof}
Since $\left\{X_n\right\}$ is exponentially tight, it satisfies the exponential compact containment condition. Thus,  there exists  a compact set $K_a$, such that 
\[\limsup_n \f{1}{n} \log P\left[X_n(t) \notin K_a, \ \T{for some} \  t < a \right] \leq -a/3.\]
Let $\tau_{n,a} = \inf\left\{s: X_n(s) \notin K_a\right\}$. Then notice that
\[P\left[\tau_{n,a} < a\right] = P\left[X_n(t) \notin K_a, \ \T{for some} \  t < a \right]. \]
Hence, 
\begin{align}
\limsup_n \f{1}{n} \log P\left[\tau_{n,a} < a\right] \leq -a/3.
\label{tna1}
\end{align}

\np
Since  $S_N(x) \equiv \sum_{k=1}^N p_k(x)\phi_k\rt x$ as $N\rt \infty$  uniformly over compact sets (see Lemma \ref{cptconv} in Appendix),
 choose $N>0$ such that
$\sup_{x\in K_a}\|S_N(x) -x\| \leq \e$, where $\e>0$ is to be specified later.\\
Therefore
\begin{align}
\sup_{s<\tau_{n,a}}\|S_N(X_n(s)) -X_n(s)\| \leq \e.
\label{tauna}
\end{align}
Observe that,
\begin{align}
\non
d(S_N(X_{n-})\cdot Y_n, X_{n-}\cdot Y_n)& \leq \int_0^\infty  e^{-u}\sup_{t\geq 0} |S_N(X_{n-})\cdot Y_n(t\wedge u) - X_{n-}\cdot Y_n(t\wedge u)| \wedge 1 \ du\\ 
\non
& \leq \int_0^{\tau_{n,a}}(\hdots) + \int_{\tau_{n,a}}^\infty (\hdots) \\
& \leq \sup_{s<\tau_{n,a}} |S_N(X_{n-})\cdot Y_n(s) - X_{n-}\cdot Y_n(s)|+  e^{-\tau_{n,a}}.
\label{equn1}
\end{align}
Note that by our notation,
\[P_N(X_{n-})\cdot Y_n(\alpha_N,\cdot) = S_N(X_{n-})\cdot Y_n =\sum_{j=1}^N \int p_j(X_{n}(s-)) \ dY_k(\phi_k,s). \]
Let $t>0$ be such that $e^{-t}<\eta$, and notice that for any $t>0$, by  the definition of the metric $d$ in \cite[Page 117]{EK86},
\begin{align}
d(P_N(X_{n-})\cdot Y_n(\alpha_N,\cdot), (P_N(X_{n-}))_\delta\cdot Y_n(\alpha_N,\cdot)) & \leq \sup_{s\leq t}|(P_N(X_{n-}) - (P_N(X_{n-}))_\delta)\cdot Y_n(\alpha_N,\cdot)(s)|
  + e^{-t}.
\label{equn2}
\end{align}
Thus, using (\ref{equn1}) and (\ref{equn2}), 
\begin{align*}
d(X_{n-}\cdot Y_n,(P_N(X_{n-}))_\delta\cdot Y_n(\alpha_N,\cdot))& \leq \sup_{s<\tau_{n,a}} |S_N(X_{n-})\cdot Y_n(s) - X_{n-}\cdot Y_n(s)|+  e^{-\tau_{n,a}}\\
& \ \ \ + \sup_{s\leq t}|(P_N(X_{n-}) - (P_N(X_{n-}))_\delta)\cdot Y_n(\alpha_N,\cdot)(s)| + e^{-t}.
\end{align*}
Hence,
\begin{align*}
\left\{d(X_{n-}\cdot Y_n, (P_N(X_{n-}))_\delta\cdot Y_n(\alpha_N,\cdot))>2\eta\right\} &\subset \left\{\sup_{s<\tau_{n,a}} |S_N(X_{n-})\cdot Y_n(s) - X_{n-}\cdot Y_n(s)|>\eta/3\right\}\\
& \hspace{.4cm}\cup\left\{e^{-\tau_{n,a}}>\eta/3\right\}\\
& \hspace{.4cm} \cup \le\{\sup_{s\leq t}|(P_N(X_{n-}) - (P_N(X_{n-}))_\delta)\cdot Y_n(\alpha_N,\cdot)(s)|
 >\eta/3\ri\}.
\end{align*}
Let $a>-\log \eta/3$. Choose $\e=\delta$ such that
\[\e k(t,a/3) <\eta/3,\]
where $k(t,a)$ was defined in \eqref{UET0}.
Then using (\ref{tauna}) and the uniform exponential tightness of $\left\{Y_n\right\}$ (\ref{UET0}), it follows that
\[\limsup_{n\rt\infty}\f{1}{n}\log P\left[d(X_{n-}\cdot Y_n, (P_N(X_{n-}))_\delta\cdot Y_n(\alpha_N,\cdot))>\eta\right] \leq -a.\]
\end{proof}

\begin{theorem} \label{th:rateidH}
Suppose for $(x, y,z) \in D_{\H\times\R^{\infty}\times\R}[0,\infty)$, $\y$ defined by (\ref{ydef2})  $\in \SC{D}$. 
\begin{enumerate}[label={\rm (\roman*)}, leftmargin=*, align=right]
\item \label{assert1} If $z \neq x\cdot \y$, then
\[J_\alpha(x, y,z) =\infty.\]
\item \label{assert2} If $z = x\cdot \y$, then
\[J_\alpha(x, y,z) =I_\alpha(x, y).\]
\end{enumerate}
\end{theorem}

\begin{proof}
\ref{assert1}
Fix an $a > 0$. By the hypothesis, there exists an $\eta>0$ such that 
\begin{align}
\label{eta}
d(z,x\cdot\y) > 4\eta >0.
\end{align}
For this $\eta$, choose $N>0$ and $\delta>0$ such that the conclusion of Lemma \ref{approx0} is satisfied. In fact, choose $N>0$ large enough, and $\delta>0$ small enough, so that
\[d(x\cdot \y, (P_N(x))_\delta\cdot \y(\alpha_N,\cdot)) <2\eta,\]
where $\y(\alpha_N,\cdot) \equiv (\y(\phi_1,\cdot),\hdots,\y(\phi_N,\cdot)) = (y_1,\hdots,y_N)=  y^{(N)}$.
Then,
\[d(z,(P_N(x))_\delta\cdot \y(\alpha_N,\cdot)) > 2\eta.\]
Define a function $G^\delta: D_{\H\times\R^\infty}[0,\infty) \rt D_{\R}[0,\infty)$ by
\[G^\delta(h, u) = (P_N(h))_\delta\cdot  u^{(N)}. \]
Take $\delta$ smaller (if necessary) so that $G^\delta$ is continuous at $(x, y)$.\\
Then there exists an $r>0$ such that
\begin{align}
\label{ball}
G^\delta(B_{r}(x,y)) \subset B_{\eta}(x\cdot\y).
\end{align}
Next, observe that
\[\left\{(X_n, Y_n(\alpha,\cdot)) \in B_{r}(x, y), X_{n-}\cdot Y_n \in B_{\eta}(z), d((P_N(X_{n-}))_\delta\cdot Y_n(\alpha_N,\cdot),X_{n-}\cdot Y_n) \leq 2\eta\right\} = \emptyset,\]
as otherwise, by (\ref{ball})
\begin{align*}
d((P_N(X_{n-}))_\delta\cdot Y_n(\alpha_N,\cdot), x\cdot\y) &<\eta. \\
\end{align*}
It follows that
\begin{align*}
d(z,x\cdot\y)& < d(z,X_{n-}\cdot Y_n) + d((P_N(X_{n-}))_\delta\cdot Y_n(\alpha_N,\cdot),X_{n-}\cdot Y_n)\\
& \hspace{.4cm}+ d((P_N(X_{n-}))_\delta\cdot Y_n(\alpha_N,\cdot), x\cdot\y) \ \ < 4\eta,
\end{align*}
which is a contradiction to (\ref{eta}). Thus, 
\begin{align*}
\{(X_n, Y_n(\alpha,\cdot)) \in B_{r}(x, y), X_{n-}\cdot Y_n \in B_{\eta}(z)\}&=\{(X_n, Y_n(\alpha,\cdot)) \in B_{r}(x, y), X_{n-}\cdot Y_n \in B_{\eta}(z),\\
& \ \ \ \ d((P_N(X_{n-}))_\delta\cdot Y_n(\alpha_N,\cdot),X_{n-}\cdot Y_n) > 2\eta\}.
\end{align*}
By Lemma \ref{approx0} 
\begin{align*}
-J_{\alpha}(x, y,z) &\leq \limsup_{n\rt\infty}\f{1}{n}\log P\left[(X_n, Y_n(\alpha,\cdot)) \in B_{r}(x, y), X_{n-}\cdot Y_n \in B_{\eta}(z)\right]\\
&\leq \limsup_{n\rt\infty}\f{1}{n} \log P\left[d((P_N(X_{n-}))_\delta\cdot Y_n(\alpha_N,\cdot),X_{n-}\cdot Y_n) > 2\eta\right] \leq -a.
\end{align*}
Since this is true for all  $a>0$, we are done.

\vspace{.3cm}
\ref{assert2}
 Since $T_t(\y)<\infty$ and $z = x\cdot \y$, for every $\eta>0$ we can find $N>0$ such that
\[d((x, y,z),(x, y,P_N(x)\cdot y^{(N)})) <\eta.\]
From (\ref{findim1}) and (\ref{findim2}),
\[J_\alpha^N(x, y,P_N(x)\cdot y^{(N)})) = I_\alpha(x, y).\]
Recall that
$$J_{\alpha}(x, y,z) = \lim_{\eta\rt 0}\liminf_{m\rt\infty}J^m_{\alpha}(B_{\eta}((x, y,z))).$$
It follows that
$$J_{\alpha}(x, y,z) \leq I_\alpha(x, y).$$
Since the reverse inequality is always true by the contraction principle, the theorem follows. 
\end{proof}

We summarize our results in the following theorem.
\begin{theorem}
\label{th_ldp_si}
Let $\H$ be a separable Banach space and $\left\{(\phi_{k},p_{k})\right\}$  a pseudo-basis of $\H$ satisfying \ref{2} of Theorem \ref{hom}, or a Schauder basis, if the latter exists. Let $\left\{Y_n\right\}$ be a sequence of $\left\{\SC{F}^n_t\right\}$-adapted, standard $\H^\#$-semimartingales and $\left\{X_n\right\}$  a sequence of cadlag, adapted $\H$-valued processes. Assume that $\left\{Y_n\right\}$ is UET. If $\left\{(X_n,Y_n)\right\}$ satisfies a LDP in the sense of Definition \ref{LDP_def}, then $\left\{(X_n, Y_n, X_{n-}\cdot Y_n)\right\}$ also satisfies a LDP. The associated rate function of the tuple $\left\{(X_n,Y_n(\alpha,\cdot), X_{n-}\cdot Y_n)\right\}$ can be expressed as
\begin{eqnarray}
J_{\alpha}(x, y, z)& = 
\begin{cases}
I_{\alpha}(x,y), & z = x\cdot \y, \ \ \  \y \in\SC{D};\\
\infty, & \mbox{otherwise,}
\end{cases}
\end{eqnarray}
where $\alpha=(\phi_1,\phi_2,\hdots,), Y_n(\alpha,\cdot) = (Y_n(\phi_1,\cdot), Y_n(\phi_2,\cdot),\hdots)$, $\y$ and $\SC{D}$ are defined by (\ref{ydef2}) and (\ref{DHdef}) respectively.
\end{theorem}

\subsection{LDP results for stochastic integrals: $(\LL, \hat{\H})^\#$-semimartingale}
Analogous to Theorem \ref{LDP_mainth}, we have our theorem on the large deviation principle of the stochastic integrals with respect to a sequence of $(\LL, \hat{\H})^\#$-semimartingale. The development of the proof is almost exactly similar to that of Theorem \ref{LDP_mainth}. 
\begin{theorem}
\label{LDP_mainth2}
Let $\left\{Y_n\right\}$ be a sequence of $\left\{\SC{F}^n_t\right\}$-adapted, standard $(\LL, \hat{\H})^\#$-semimartingales and $\left\{X_n\right\}$ be a sequence of cadlag, adapted $\hat{\H}$- valued processes. Assume that $\left\{Y_n\right\}$ is UET. If $\left\{(X_n,Y_n)\right\}$ satisfies a LDP in the sense of Definition \ref{LDP_def} with the  rate function family $\left\{I_\alpha:\alpha \in A\right\}$, then $\left\{(X_n, Y_n, X_{n-}\cdot Y_n)\right\}$ also satisfies a LDP.
\end{theorem}

\subsubsection{Identification of the rate function}
As before, we follow the process of identification of the rate function of the tuple $\left\{(X_n, Y_n, X_{n-}\cdot Y_n)\right\}$ in Theorem \ref{LDP_mainth2}.
Again we assume that $\H$ is a separable Banach space with a pseudo-basis $\left\{(\phi_k,p_k)\right\}$ satisfying \ref{2} of Theorem \ref{hom} . The notation used in the beginning of Section \ref{sec:rateid} is  used here as well.
The following theorem is the analogue of Theorem \ref{UETprop} and the proof is almost exactly the same.

\begin{theorem}
\label{UETprop2}
Let $\H$ be a separable Banach space. Choose a pseudo-basis $\left\{(\phi_{k},p_{k})\right\}$ of $\H$ satisfying \ref{2} of Theorem \ref{hom}, or a Schauder basis, if the latter exists. Suppose that $\{Y_n\}$ is a UET sequence. Suppose that for $(x,y) \in D_{\hat{\H}\times \R^\infty}[0,\infty)$,  $I_\alpha(x,y)<\infty$. Then,
\begin{enumerate}[label={\rm (\roman*)}, leftmargin=*, align=right]
\item \label{yinh2_2} for every compact set $C \subset \H$,
$$\sup_{t\leq T}\sup_{h\in C}|\sum_{j=M}^{N}y_j(t)p_j(h)| \rt 0, \ \ \ \mbox{as} \ \ M, N \rt \infty;$$
\item \label{ylinear2} for every $t>0$, $\y(t) \in \H^*$, where $\y(t)\equiv\sum_{j} y_j(t)p_j$;
\item \label{yfv_2} for every $t>0$, $T_t(\y) < \infty$.
\end{enumerate}
\end{theorem}

We next state an approximation result similar to Lemma \ref{approx0}, and the key in this more general setting is to work with the right kind of approximation scheme. 
Recall that $\hat{\H}$ is the completion of the linear space $\hat{\SC{H}}\equiv\left\{\sum_{i=1}^l\sum_{j=1}^ma_{ij}f_i\phi_j: f_i \in \left\{f_i\right\}, \phi_j \in \left\{\phi_j\right\}\right\}$ with respect to a suitable norm $\|\cdot\|_{\hat{\H}}$ (c.f. Definition \ref{Hhat}).\\
For $h = \sum_{i=1}^L\sum_{j=1}^Ma_{ij}f_i\phi_j$, define
$Q_{l,m}:\hat{\SC{H}} \rt \LL^m$ by
$$Q_{l,m}(h)=(\sum_{i=1}^la_{i1}f_i,\sum_{i=1}^la_{i2}f_i,\hdots,\sum_{i=1}^la_{im}f_i).$$
For  $g\in D_{\LL^k}[0,\infty)$ define
\begin{align*}
g_\delta (t) = \sum_{k}g(\tau_k^\delta)1_{[\tau_k^\delta,\tau_{k+1}^\delta)}(t),
\end{align*}
where the $\tau_k^\delta$ are defined by
\begin{align*}
\tau_0^\delta & = 0,\\
\tau_{k+1}^\delta & = \inf\left\{s>\tau_k^\delta: \|g(s) - g(\tau_k^\delta)\|_\LL >\delta\right\}.
\end{align*}
For $y \in D_{\R^k}[0,\infty) $, define that,
\[g_\delta\cdot y(t) = \sum_k g(\tau_k^\delta)^T(y(\tau_{k+1}^\delta\wedge t) - y(\tau_k^\delta\wedge t)). \]
where for $(f_1,\hdots,f_m)^T \in \LL^k$ and $(y_1,\hdots,y_k)^T \in \R^k$, 
$$(f_1,\hdots,f_m)(y_1,\hdots,y_k)^T = \sum_iy_if_i \in \LL.$$

\begin{lemma}
\label{approx1}
For $\eta>0$ and $a>0$, there exist sufficiently small $\e>0$, $\delta>0$ and sufficiently large  $l,m>0$ such that
\[\limsup_{n\rt\infty}\f{1}{n}\log P\left[d(X_{n-}\cdot Y_n, (Q_{l,m}(X^\e_{n-}))_\delta\cdot Y_n(\alpha_N,\cdot))>2\eta\right] \leq -a,\]
where $X_n^\e(s) = \sum_{i,j}^{l,m}c_{ij}^\e(X_n(s))f_i\phi_j$, where recall that the $c_{ij}$ are defined by \eqref{xep}.
\end{lemma}

\begin{proof}
Since $\left\{X_n\right\}$ is exponentially tight, it satisfies the exponential compact containment condition. Thus,  there exists  a compact set $K_a \subset \hat\H$ such that 
\[\limsup_n \f{1}{n} \log P\left[X_n(t) \notin K_a, \ \T{for some} \  t < a \right] \leq -a/3.\]
Let $\tau_{n,a} = \inf\left\{s: X_n(s) \notin K_a\right\}$. Then notice that
\[P\left[\tau_{n,a} < a\right] = P\left[X_n(t) \notin K_a, \ \T{for some} \  t < a \right]. \]
Hence, 
\begin{align*}
\limsup_n \f{1}{n} \log P\left[\tau_{n,a} < a\right] \leq -a/3.
\end{align*}

\np
Fix $\e>0$. Then there exists $l,m>0$ such that  defining 
$x^\e = \sum_{i,j}^{l,m}c_{ij}^\e(x)f_i\phi_j,$
we have
$$\|x^\e-x\|_{\hat \H}\leq \e,\ \mbox{ for all } x \in K_a.$$
Hence,
\begin{align*}
\sup_{s<\tau_{n,a}}\|X_n^\e(s) -X_n(s)\|_{\hat \H} \leq \e.
\end{align*}
and the rest of the proof follows similar steps used after \eqref{tauna} in the proof of Lemma \ref{approx0}, except that we now use $Q_{l,m}$ and $X_n^\e$ instead of $P_N$ and $S_N(X_n)$.
\end{proof}

The following result generalizes Theorem \ref{th:rateidH} to the wider class of $(\LL, \hat{\H})^\#$-semimartingales. We only present the key steps.
\begin{theorem}\label{th:rateidLH}
Suppose for $(x, y,z) \in D_{\hat{\H}\times\R^{\infty}\times\LL}[0,\infty)$, $\y$ defined by (\ref{ydef2})  $\in \SC{D}$. 
\begin{enumerate}[label={\rm (\roman*)}, leftmargin=*, align=right]
\item \label{Assert1_lh} If $z \neq x\cdot \y$, then
\[J_\alpha(x, y,z) =\infty.\]
\item \label{Assert2_lh} If $z = x\cdot \y$, then
\[J_\alpha(x, y,z) =I_\alpha(x, y).\]
\end{enumerate}
\end{theorem}

\begin{proof}
\ref{Assert1_lh} Fix an $a>0$. By the hypothesis, there exists an $\eta>0$ such that 
\begin{align}
\label{eta2}
d(z,x\cdot\y) > 4\eta >0.
\end{align}
For this $\eta$, choose $l,m>0$ and $\e, \delta>0$ such that conclusion of Lemma \ref{approx1} is satisfied. In fact, choose $l,m>0$ large enough and $\delta>0$ small enough so that
\[d(x\cdot \y, (Q_{l,m}(x^\e))_\delta\cdot \y(\alpha_m,\cdot)) <2\eta,\]
where $x^\e = \sum_{i,j}^{l,m}c_{ij}^\e(x)f_i\phi_j $, $\y(\alpha_m,\cdot) \equiv (\y(\phi_1,\cdot),\hdots,\y(\phi_m,\cdot)) = (y_1,\hdots,y_m)=  y^{(m)}$.
Then
\[d(z,(Q_{l,m}(x^\e))_\delta\cdot \y(\alpha_m,\cdot)) > 2\eta.\]
Now following the steps used in Theorem \ref{th:rateidH}, we define a function $G^{\delta,\e}: D_{\hat{\H}\times\R^\infty}[0,\infty) \rt D_{\R}[0,\infty)$ by
$G^{\delta,\e}(h, u) = (Q_{l,m}(h^\e))_\delta\cdot  u^{(m)}, $
and the rest of the proof is similar.

\vspace{.3cm}
\ref{Assert2_lh}
 Since $T_t(\y)<\infty$ and $z = x\cdot \y$, for every $\eta>0$, we can find sufficiently large $l,m>0$ such that
\[d((x, y,z),(x, y,Q_{l,m}(x)\cdot y^{(m)})) <\eta.\]
Let $J_\alpha^{l,m}$ denote the rate function for $(X_n,Y_n(\alpha,\cdot), Q_{l,m}(X_{n-})\cdot Y_n(\alpha_m,\cdot) )$ and note that by our notation
\[Q_{l,m}(X_{n-})\cdot Y_n(\alpha_m,\cdot)  =\sum_{i=1}^l\sum_{j=1}^m f_i\int c_{i,j}(X_{n}(s-)) \ dY_k(\phi_j,s). \]
Similar to \eqref{findim1} and \eqref{findim2}, in this case the finite-dimensional result of Garcia \cite{Gar08} yields that
\[J_\alpha^{l,m}(x, y,Q_{l,m}(x)\cdot y^{(m)})) = I_\alpha(x, y).\]
Now just like \eqref{rate3}
$$J_{\alpha}(x, y,z) = \lim_{\eta \rt 0} \liminf_{l,m\rt\infty}J^{l,m}_{\alpha}(B_{\eta}((x, y,z))),$$
and it follows that
$$J_{\alpha}(x, y,z) \leq I_\alpha(x, y).$$
As the reverse inequality is always true by the contraction principle, the theorem follows. 
\end{proof}

We summarize our results in the following theorem.
\begin{theorem}
\label{th_ldp_si_2}
Let $\H$ be a separable Banach space and $\left\{(\phi_{k},p_{k})\right\}$  a pseudo-basis of $\H$ satisfying \ref{2} of Theorem \ref{hom}, or a Schauder basis, if the latter exists. Let $\left\{Y_n\right\}$ be a sequence of $\left\{\SC{F}^n_t\right\}$-adapted, standard  $(\LL, \hat{\H})^\#$-semimartingales and $\left\{X_n\right\}$ a sequence of cadlag, adapted $\hat{\H}$ valued processes. Assume $\left\{Y_n\right\}$ is UET. If $\left\{(X_n,Y_n)\right\}$ satisfies a LDP in the sense of Definition \ref{LDP_def}, then $\left\{(X_n, Y_n, X_{n-}\cdot Y_n)\right\}$ also satisfies a LDP.  The associated rate function of the tuple $\left\{(X_n,Y_n(\alpha,\cdot), X_{n-}\cdot Y_n)\right\}$ can be expressed as
\begin{eqnarray}
J_{\alpha}(x, y, z)& = 
\begin{cases}
I_{\alpha}(x,y), & z = x\cdot \y, \ \ \  \y \in\SC{D},\\
\infty, & \mbox{otherwise,}
\end{cases}
\end{eqnarray}
where $\alpha=(\phi_1,\phi_2,\hdots,), Y_n(\alpha,\cdot) = (Y_n(\phi_1,\cdot), Y_n(\phi_2,\cdot),\hdots)$, $\y$ and $\SC{D}$ are defined by (\ref{ydef2}) and (\ref{DHdef}) respectively.
\end{theorem}

\setcounter{equation}{0}
\renewcommand {\theequation}{\arabic{section}.\arabic{equation}}
\section{LDP for stochastic differential equation}
\label{sec:SDE}
\subsection{$\H^\#$-semimartingale}
\begin{theorem}
\label{sde}
Let $\H$ be a separable Banach space and $\left\{(\phi_{k},p_{k})\right\}$  a pseudo-basis of $\H$ satisfying \ref{2} of Theorem \ref{hom}, or a Schauder basis, if the latter exists. Let $\left\{Y_n\right\}$ be a sequence of uniformly exponentially tight $\left\{\SC{F}^n_t\right\}$-adapted, standard $\H^\#$-semimartingales, $\left\{X_n\right\}$ a sequence of cadlag, adapted $\R^d$-valued processes and $\left\{U_n\right\}$ a sequence of adapted $\R^d$-valued cadlag processes. Suppose that $\left\{(U_n, Y_n)\right\}$ satisfies a large deviation principle with the rate function family $\left\{I_\beta(\cdot,\cdot)\right\}$. Assume that $F: \R^d \rt \H$ is a continuous function and  $X_n$ satisfies
\[X_n(t) = U_n(t)+F(X_{n-})\cdot Y_n(t).\]
For $ y \in D_{\R^\infty}[0,\infty)$, define $\y$ by (\ref{ydef2}). Suppose that for every $(u,y) \in D_{\R^d\times\R^\infty}[0,\infty)$ for which $I_\alpha(u,y) < \infty$, the solution to 
$$x = u + F(x)\cdot \y$$
is unique.
Assume further that  $\left\{(U_n, X_n,Y_n)\right\}$ is exponentially tight.
Then the sequence \\$\left\{(U_n,X_n,Y_n(\alpha,\cdot))\right\}$ satisfies a LDP in $D_{\R^d\times \R^d\times\R^\infty}[0,\infty)$ with the rate function given by
\begin{eqnarray}
J_{\alpha}(u,x, y)& = 
\begin{cases}
I_{\alpha}(u,y), & x = u + F(x)\cdot \y, \ \ \  \y \in\SC{D},\\
\infty, & \mbox{otherwise},
\end{cases}
\end{eqnarray}
where  $\y$ and $\SC{D}$ are defined by (\ref{ydef2}) and (\ref{DHdef}) respectively.
\end{theorem}

\begin{proof} We borrow some ideas from the proof of Theorem 8.2 of Garcia \cite{Gar08}.
To prove that $\left\{(U_n,X_n,Y_n(\alpha,\cdot))\right\}$ satisfies a large deviation principle, its enough to prove that for every subsequence of $\left\{(U_n,X_n,Y_n(\alpha,\cdot)\right\}$, there exists a further subsequence which satisfies a LDP with the same rate function $J_{\alpha}$. Since for every exponentially tight sequence there exists a subsequence which satisfies a LDP, we can assume that $\left\{(U_n,X_n,Y_n(\alpha,\cdot))\right\}$ satisfies a LDP with the rate function $J_{\alpha}$,  and then show that the expression of $J_{\alpha}$ does not depend on the choice of a subsequence.\\
Let $(u,x, y) \in \ D_{R^d\times\R^d\times\R^\infty}[0,\infty)$.
Notice that if $\y$ defined by (\ref{ydef2}) is not in $\SC{D}$, then $J_{\alpha}(u,x,y) = \infty$ by Theorem \ref{yinh} and \ref{yfv}.\\
Next suppose that $\y \in \SC{D}$, but $x\neq u+F(x)\cdot \y$. We will prove that in this case
\[J_\alpha(u,x, y) = \infty.\]
Notice that by the contraction principle, $\left\{(U_n, F(X_n),Y_n(\alpha,\cdot))\right\}$ satisfies a LDP with the rate function
$$ \tilde{\Lambda}_\alpha(u, z, y) = \inf_{q}\left\{J_\alpha(u,q, y): F(q)=z\right\}.$$
It follows from Theorem \ref{th_ldp_si} that $\left\{(U_n, F(X_n),Y_n(\alpha,\cdot), F(X_{n-})\cdot Y_n)\right\}$ satisfies a  LDP with the rate function 
\begin{equation*}
\Lambda_\alpha(u, z, y,p) =
\begin{cases}
 \inf_{q}\left\{J_\alpha(u,q,y): F(u)=z\right\}, & p= z\cdot\y, \ \  \y \in \SC{D},\\
\infty, & \mbox{otherwise.}
\end{cases}
\end{equation*}
Since $X_n = U_n+F(X_{n-})\cdot Y_n$, it now follows from the contraction principle that  $\left\{(U_n, F(X_n),X_n, Y_n(\alpha,\cdot))\right\}$ satisfies the LDP with the rate function given by
\begin{equation}
\label{ratej}
\tilde{J}_\alpha(u,z, x, y) =
\begin{cases}
 \inf_q\left\{J_\alpha(u,q, y): F(q)=z\right\}, & x= u+z\cdot\y, \ \  \y \in \SC{D},\\
\infty, & \mbox{otherwise.}
\end{cases}
\end{equation}
On the other hand, notice that since $\left\{(U_n, X_n, Y_n(\alpha,\cdot))\right\}$ satisfies LDP with rate function $J_\alpha$, the contraction principle yields 
\begin{equation}
\label{ratej2}
 \tilde{J}_\alpha(u, z, x, y) = 
\begin{cases}
 J_\alpha(u, x,y), & \mbox{ if } z = F(x),\\
 \infty, & \mbox{otherwise.}
\end{cases}
\end{equation}

\np
We will prove that if $x\neq u + F(x)\cdot \y$, then for all $z$ 
\[\tilde{J}_\alpha(u, z,x, y) = \infty.\]
Then taking infimum over all $z$, we see that $J_\alpha(u,x, y) = \infty$.\\

\np
{\it Case 1:} $x\neq u + F(x)\cdot \y$ \\
Fix a $z$. Notice that if $x \neq u + z\cdot \y$,  $\tilde{J}_\alpha(u,z,x, y) =\infty$. So suppose that $x = u + z\cdot \y$. Then from the assumption we find that
$z\neq F(x)$. It follows from (\ref{ratej2}) that, $\tilde{J}_\alpha(u,z,x, y) = \infty$.\\

\np
Next assume that $x= u + F(x)\cdot \y$. We will prove that $J_\alpha(u,x, y) = I_\alpha(u, y)$.\\
{\it Case 2:} $x= u + F(x)\cdot \y$\\
If $I_{\alpha}(u,y) =\infty$, clearly $J_{\alpha}(u,x,y) =\infty$. So, assume that $I_{\alpha}(u,y) < \infty$. Then there exists a $q$ such that $J_\alpha(u,q,y)<\infty$.
From  {\it Case 1} it follows that $q =u + F(q)\cdot \y$. By the uniqueness assumption, it follows that $q=x$. Hence we have
$$J_\alpha(u,x,y) = J_\alpha(u,q,y), \T{ for all $q$ such that } J_\alpha(u,q,y)<\infty.$$
It follows that
$$J_\alpha(u,x,y) = \inf_qJ_\alpha(u,q,y) = I_\alpha(u,y).$$
\end{proof}

\np
The above theorem can be extended to cover stochastic differential equations of the type
$$X_n(t) = U_n+F_n(X_{n-})\cdot Y_n,$$
where the $F_n:\R^d\rt \H$ are measurable functions satisfying some suitable conditions.

\begin{theorem}
\label{sde2}
Let $\H$ be a separable Banach space and $\left\{(\phi_{k},p_{k})\right\}$  a pseudo-basis of $\H$ satisfying \ref{2} of Theorem \ref{hom}, or a Schauder basis, if the latter exists. Let $\left\{Y_n\right\}$ be a sequence of uniformly exponentially tight $\left\{\SC{F}^n_t\right\}$-adapted, $\H^\#$-semimartingales, $\left\{X_n\right\}$ a sequence of cadlag, adapted $\R^d$-valued processes and $\left\{U_n\right\}$ a sequence of adapted $\R^d$-valued cadlag processes. Suppose $\left\{(U_n, Y_n\right\}$ satisfies large deviation principle with the rate function family $\left\{I_\beta(\cdot,\cdot)\right\}$. Assume that $F, F_n: \R^d \rt \H$ are measurable functions such that
\begin{itemize}
\item for all $x$ whenever $x_n\rt x$, $F_n(x_n) \rt F(x).$  
\end{itemize}
Suppose that $X_n$ satisfies
\[X_n(t) = U_n(t)+F_n(X_{n-})\cdot Y_n(t).\]
For $ y \in D_{\R^\infty}[0,\infty)$, define $\y$ by (\ref{ydef2}). Suppose that for every $(u,y) \in D_{\R^d\times\R^\infty}[0,\infty)$ for which $I_\alpha(u,y) < \infty$, the solution to 
$$x = u + F(x)\cdot \y$$
is unique.
Assume further that  $\left\{(U_n, X_n,Y_n)\right\}$ is exponentially tight.
Then the sequence $\left\{(U_n,X_n,Y_n(\alpha,\cdot))\right\}$ satisfies a LDP with the rate function given by
\begin{eqnarray}
J_{\alpha}(u,x, y)& = 
\begin{cases}
I_{\alpha}(u,y), & x = u + F(x)\cdot \y, \ \ \  \y \in\SC{D},\\
\infty, & \mbox{otherwise.}
\end{cases}
\end{eqnarray}
\end{theorem}

\begin{proof}
The proof is almost exactly same as the above theorem once we apply a generalized version of the contraction principle (see Theorem \ref{cont_mod}) instead of the usual one. 
\end{proof}

\subsection{ $(\LL, \hat{\H})^\#$-semimartingale}
We now generalize Theorem \ref{sde2} to cover   infinite-dimensional SDEs driven by  $(\LL, \hat{\H})^\#$-semimartingales.

\begin{theorem}
	\label{sde3}
	Let $\H$ be a separable Banach space and $\left\{(\phi_{k},p_{k})\right\}$  a pseudo-basis of $\H$ satisfying \ref{2} of Theorem \ref{hom}, or a Schauder basis, if the latter exists. Let $\left\{Y_n\right\}$ be a sequence of uniformly exponentially tight $\left\{\SC{F}^n_t\right\}$-adapted, $(\LL, \hat{\H})^\#$-semimartingales, $\left\{X_n\right\}$ a sequence of cadlag, adapted $\LL$-valued processes and $\left\{U_n\right\}$ a sequence of adapted $\LL$-valued cadlag processes. Suppose $\left\{(U_n, Y_n\right\}$ satisfies large deviation principle with the rate function family $\left\{I_\beta(\cdot,\cdot)\right\}$. Assume that $F, F_n: \LL \rt \hat{\H}$ are measurable functions such that
	\begin{itemize}
		\item for all $x$ whenever $x_n\rt x$, $F_n(x_n) \rt F(x).$  
	\end{itemize}
	Suppose that $X_n$ satisfies
	\[X_n(t) = U_n(t)+F_n(X_{n-})\cdot Y_n(t).\]
	For $ y \in D_{\R^\infty}[0,\infty)$, define $\y$ by (\ref{ydef2}). Suppose that for every $(u,y) \in D_{\LL\times\R^\infty}[0,\infty)$ for which $I_\alpha(u,y) < \infty$, the solution to 
	$$x = u + F(x)\cdot \y$$
	is unique.
	Assume further that  $\left\{(U_n, X_n,Y_n)\right\}$ is exponentially tight.
	Then the sequence $\left\{(U_n,X_n,Y_n(\alpha,\cdot))\right\}$ satisfies a LDP in $D_{\LL\times \LL\times \R^{\infty}}[0,\infty)$ with the rate function given by
	\begin{eqnarray}
		J_{\alpha}(u,x, y)& = 
		\begin{cases}
			I_{\alpha}(u,y), & x = u + F(x)\cdot \y, \ \ \  \y \in\SC{D},\\
			\infty, & \mbox{otherwise.}
		\end{cases}
	\end{eqnarray}
\end{theorem}
The proof follows almost the  exact same route  as that of Theorem \ref{sde}.

\setcounter{equation}{0}
\renewcommand {\theequation}{\arabic{section}.\arabic{equation}}
\section{Examples}\label{sec:ex}
The above results lay out a systematic program for validation of LDP for SDEs: (i) verify that the driving semimartingale sequence $\{Y_n\}$ is UET, (ii) check that $\{Y_n\}$ satisfies LDP (usually follows some standard results) and identify the rate function, (iii) prove that $X_n$ is exponentially tight. The desired result on the LDP for $\{X_n\}$ along with its rate function then readily follows from Theorems \ref{sde2} or  \ref{sde3}.  As mentioned before, the choice of the indexing space $\H$ for the  driving infinite-dimensional semimartingales $Y_n$ smartly encodes some of the necessary conditions required on the coefficients of the SDEs.

\begin{example}[LDP for Markov chains] \label{Markovsde}{\rm   

Let $\left\{X^n_k\right\}$ be a Markov Chain in  $\R^d$ satisfying
 $$X^n_{k+1} = X^n_k + \f{1}{n}b(X^n_k,\xi_{k+1}), \quad X^n_0 = x_0,$$
where the $\xi_k$ are iid random variables in $\E$ with distribution $\pi$ and $b:\R^d\times \E \rt \R^d.$
We want a LDP for $\left\{X^n(t) \equiv X^n_{\left[nt\right]}\right\}$.
Define 
$$M_n(\Gamma, t) = \sum_{k=1}^{\left[nt\right]}1_{\Gamma}(\xi_k),$$
and notice that,
$$X^n(t) = x_0 + \f{1}{n}\int_{\E\times[0,t)} b(X^n(s),u)M_n(du\times ds).$$
Now $M_n$ is a counting measure with mean measure $\pi\ot\mu_n$, where $\mu_n\left[0,t\right] =\left[nt\right]$. 

 Notice that the $Y_n \equiv M_n/n$ can be considered as cadlag processes on $\SC{M}_F(\E)$, that is, 
$Y_n \in D_{\SC{M}_F(\E)}[0,\infty).$
Write
\begin{align}\label{def:LE}
\SC{L}(\E) = \left\{z \in \SC{M}_F(\E\times [0,\infty)): z(\E\times \left[0,t\right]) = t, t\geq 0\right\}.
\end{align}
Topologize $\SC{L}(\E)$ by weak convergence on bounded intervals, that is, $z_n\rt z$ if
$$\int_{\E\times \left[0,t\right]}f(u,s)z_n(du\times ds) \rt \int_{\E\times \left[0,t\right]}f(u,s)z(du\times ds), \ \ t\geq 0$$
for all $f \in C_b(\E\times [0,\infty))$. If $z \in \SC{L}(\E)$, then there exists $\mu \in M_{\SC{P}(\E)}[0,\infty)$ such that
$$z(C \times [0,t)) = \int_0^t \mu(s)(C) \ ds.$$  
We write $\dot{z}(t) = \mu(t)$. Conversely, if $\mu \in M_{\SC{P}(\E)}[0,\infty)$, then $z$ defined by the above relation is in $\SC{L}(E).$ 
Now by Sanov's theorem,   $\left\{Y_n\right\}$ satisfies a LDP in $D_{\SC{M}_F(\E)}[0,\infty)$ with rate function
\begin{eqnarray}\label{rate_sanov}
	\bar I(z)& = 
	\begin{cases}
		\int_0^\infty R(\dot{z}(s)|\pi)\ ds, & z \in \SC{L}(\E)\\
		\infty, & \mbox{otherwise,}
	\end{cases}
\end{eqnarray}
where $R(\nu|\pi)\equiv \int _E \f{d\nu}{d\pi} \ln \le(\f{d\nu}{d\pi}\ri)  d\pi$ is the relative entropy of the measure $\nu$ with respect to $\pi$.

 We take the indexing space $\H$ to be $M^\Phi(\pi) \subset L^\Phi(\pi)$ (see \eqref{MTspace} for definition).

  Consider  $M_n$ as an $\H^\#$-semimartingale, for , with $\Phi(x) = e^x -1$. As standard, this means for $h\in M^\Phi(\pi)$, $M_n(h,\cdot)$ is defined by 
$$M_n(h,t) = \int_{\E\times[0,t} h(u)M_n(du\times ds) = \sum_{k=1}^{[nt]}h(\xi_k).$$ 
Assume that $b$ is Lipschitz in the first argument and 
$\sup_x\|b(x,\cdot)\|_\Phi < \infty$.\\ 
We now carry out the program of verification of LDP for $\{X_n\}$.\\

\np
{\it UET of $\{Y_n\}$:} The proof that $\left\{Y_n\equiv M_n/n\right\}$ is UET is  similar to that of Poisson random measure in Example \ref{ex_Poisson}.
Let $Z(u,s)$ be a cadlag process such that $\sup_{s\leq t}\|Z(\cdot,s)\|_\Phi \leq 1$.
Observe that without loss of generality we can take $Z \geq 0$ for our purpose.\\
Call  $H_n(t) \equiv  Z_-\cdot M_n (t)$. Then as in Example \ref{ex_Poisson}, apply It{\^o}'s lemma to get 
\begin{align*}
E(e^{H_n(t)})& = 1 + nE\int_0^t e^{H_n(s)}\int_\E(e^{Z(u,s)} - 1)\pi(du)\mu_n(ds) .\\
\end{align*} 
Now since $\|f\|_\Phi \leq 1 $ iff $\int\Phi(|f|)\ d\nu \leq 1$, we see from our assumption on the process $Z$ that
$\sup_{s\leq t}\int_\E(e^{Z(u,s)} - 1) d\nu(u) \leq 1$. Thus, 
\[E(e^{H_n(t)})\leq 1 + E\int_0^t e^{H_n(s)}\mu_n(ds),\]
and by Gronwall's inequality
\[E(e^{H_n(t)}) = E\le(e^{Z_-\cdot Y_n(t)}\ri) \leq e^{\left[nt\right]}.\]
Therefore
\begin{align*}
P(\sup_{s\leq t}Z_-\cdot Y_n(s) > K )& =  P(Z_-\cdot M_n(t) > nK ) =P\le(e^{Z_-\cdot M_n(t)} > e^{nK} \ri)\\
& \leq E(e^{Z_-\cdot M_n(t)}) /e^{nK} \leq e^{nt - nK}.
\end{align*}
Choosing $k(t,a) \equiv K= t+a$, we have
\[\limsup_n\f{1}{n} \log\sup P\left[\sup_{s\leq t}|Z_-\cdot Y_n(s)| > k(t,a)\right] \leq -a.\]\\

\np
{\it LDP of $\{Y_n\}$:} To prove that $\left\{Y_n\equiv M_n/n\right\}$ satisfies a LDP as $\H^\#$-semimartingales, we observe that, by the contraction principle, for any finite collection $h= (h_1,\hdots h_m)$ in $M^\Phi(\pi)$, 
$\{Y_n(h,\cdot) \equiv (Y_n(h_1,\cdot),\hdots,Y_n(h_m,\cdot))\}$
satisfies a LDP in $D_{\R^m}[0,\infty)$ with the rate function given by
\begin{align*}
	I_h(y^{(m)}) = \inf\le\{\bar I(z): y_i(t) = \int_{\E\times [0,t]} h_i(u) z(du\times ds) \ i=1,2,\hdots,m \ri\}, \quad  y^{(m)}  = (y_1,y_2,\hdots,y_m).
\end{align*}	
Thus choosing a pseudo-basis $\{(\phi_k,p_k)\}$ of $M^\Phi(\pi)$, we have that $\{Y_n(\alpha,\cdot) \equiv (Y_n(\phi_1,\cdot),Y_n(\phi_2,\cdot)),\cdots\}$ satisfies a LDP in $D_{R^\infty}[0,\infty)$with 
 with the rate function given by
 \begin{align}\label{rateY_Markov}
 	I_\alpha(y) = \inf\le\{\bar I(z): y_i(t) = \int_{E\times [0,t]} \phi_i(u) z(du\times ds), \  z \in \SC{L}(\E), \ i=1,2,\hdots\ri\}, \quad  y  = (y_1,y_2,\hdots).
 \end{align}	

\np
{\it Exponential tightness of $\{X_n\}$:} For simplicity of the calculation we will assume $d=1$. The calculations can be easily extended for higher $d$. By  It{\^o}'s lemma,
\begin{align*}
e^{nX^n(t+h)}& = e^{nX^n(t)} + \int_{\E\times[t,t+h)}e^{nX^n(s-)}(e^{b(X^n(s-),u)}-1)M_n(du\times ds).\\
\end{align*}
It follows that,
$$E\le(e^{n(X^n(t+h)-X^n(t))}\Big|\SC{F}^n_t\ri) =1 + E \le(\int_{\E\times[t,t+h)}e^{n(X^n(s-) -X^n(t))}(e^{b(X^n(s-),u)}-1)\pi(du)\mu_n(ds)\bigg|\SC{F}^n_t\ri).$$
Hence 
\begin{align*}
E(e^{n(X^n(t+h)-X^n(t))}|\SC{F}^n_t)& \leq 1+ nE \le(\int_{\E\times[0,h)}e^{n(X^n(t+s-) -X^n(t))}(e^{b(X^n(t+s-),u)}-1)\pi(du)ds\bigg|\SC{F}^n_t\ri)\\
& \leq 1 + n\sup_x\|b(x,\cdot)\|_\Phi \int_0^h E \le(e^{n(X^n(t+s-) -X^n(t))}\big|\SC{F}^n_t\ri) ds.
\end{align*}
By Gronwall's inequality 
$$E\le(e^{n(X^n(t+h)-X^n(t))}\big|\SC{F}^n_t\ri) \leq e^{n\sup_x\|b(x,\cdot)\|_\Phi h}.$$
Similarly, 
$$E\le(e^{-n(X^n(t+h)-X^n(t))}\big|\SC{F}^n_t\ri) \leq e^{n\sup_x\|b(x,\cdot)\|_\Phi h},$$
and it follows that  
$$E\le(e^{n|X^n(t+h)-X^n(t)|}\big|\SC{F}^n_t\ri) \leq 2e^{n\sup_x\|b(x,\cdot)\|_\Phi h}.$$
Now, the exponential tightness of $\left\{X^n\right\}$ follows from Feng and Kurtz \cite[Theorem 4.1]{FK06}.\\

\np
 {\it LDP and rate function of $\{X_n\}$}: Theorem \ref{sde} now readily establishes the LDP of $\left\{X^n\right\}$ with the rate function given by
$$J(x) = \inf \left\{I_{\alpha}(y): \dot{x}(s) = b(x(s),\cdot)\cdot \y(ds), \ \y \in \SC{D}\right\}.$$
where   $\y$ and $\SC{D}$ are defined by (\ref{ydef2}) and (\ref{DHdef}) respectively.
Now if $z\in \SC{L}(\E)$ and $(y_1,y_2,y_3,\hdots)$ are such that $ y_i(t) = \int_{\E\times [0,t]} \phi_i(u) z(du\times ds)$ and $ \dot{x}(s) = b(x(s),\cdot)\cdot \y(ds)$, then it is easy to see that
$$x(t) = x_0+ \int_{\E\times[0,t]}  b(x(s),u) z(du\times ds).$$
Consequently, after a few elementary steps, it follows that the rate function $J$ can be equivalently expressed as
$$J(x) = \inf \left\{\bar I(z): x(t) = x_0+ \int_{\E\times[0,t]}  b(x(s),u) z(du\times ds), \ z \in \SC{L}(E)\right\},$$
where $\bar I$ is defined in \eqref{rate_sanov}.
}
\end{example}

\vspace{1cm}

\begin{example}[LDP for random evolutions] \label{rndev}{\rm 
Let $\E$ be a complete and separable metric space. Let $\left\{\xi_k\right\}$ be a  $\left\{\SC{F}_k\right\}$-Markov chain in $\E$ with the transition kernel  $P$. Consider the evolution equation
$$X^n_{k+1} = X^n_k + \f{1}{n}b(X^n_k,\xi_{k+1}),$$
where $b:\R^d\times \E \rt \R^d.$
By a slight abuse of notation, put $\xi_n(t) = \xi_{\left[nt\right]}$, $X^n(t) = X^n_{\left[nt\right]}$.
We wish to find a LDP for $\left\{X^n\right\}$. Note that for each $n$, $\left\{X^n\right\}$ is adapted to the filtration $\left\{\SC{F}^n_t = \SC{F}_{\left[nt\right]}\right\}$.

Define 
$$M_n(\Gamma, t) = \sum_{k=1}^{\left[nt\right]}1_{\Gamma}(\xi_k),$$
and notice that $X^n$ satisfies
$$X^n(t) = X^n_0 + \f{1}{n} \int_{\E\times[0,t)} b(X^n(s),u)M_n(du\times ds).$$
Define $\tilde{M}_n$ by
$$\tilde{M}_n(\Gamma, t) = \sum_{k=1}^{\left[nt\right]}(1_{\Gamma}(\xi_k) - P(\xi_{k-1},\Gamma)).$$
Notice that for each $\Gamma$, $\tilde{M}_n(\Gamma, \cdot) $ is a martingale. Define the random measure $\mu_n$ by
$$\mu_n(\Gamma,t) =  \sum_{k=1}^{\left[nt\right]} P(\xi_{k-1},\Gamma).$$
Then,
$$E\le(\int_{\E\times [0,t)} H(u,s) M_n(du,ds)\ri) =  E(\int_{\E\times [0,t)} H(u,s) \mu_n(du,ds)).$$
Assume that there exists a $\sigma$-finite measure $\pi$  such that the Radon-Nikodym derivatives 
$$ \f{dP(x,\cdot)}{d\pi}(u) \leq C, \ \ \mbox{for all} \ x,u.$$ 

\np
{\it LDP of occupation measures of Markov chain}:
Assume that the transition matrix $P$ of the Markov chain $\left\{\xi_k\right\}$ satisfies the following uniform ergodicity condition (see Page 100, \cite{DS89}, Appendix B, \cite{FK06}): there exist $l,N \in \Z^+$ with $1\leq l \leq N$ and $M\geq 1$ such that
$$P^l(x,\cdot)\leq \f{M}{N}\sum_{m=1}^N P^m(y,\cdot), \ \mbox{ for } \ x,y \in E.$$
The above uniform ergodicity condition guarantees the existence of an unique invariant probability measure $\nu$ for $P$. 
Then $\left\{Y_n\right\}$ satisfies a LDP in $D_{\SC{M}_F(\E)}[0,\infty)$ with the rate function
\begin{eqnarray} \label{rate_Markovocc}
	\bar I(z)& = 
	\begin{cases}
		\int_0^\infty I_{P}(\dot{z}(s))\ ds, & z \in \SC{L}(\E),\\
		\infty, & \mbox{otherwise,}
	\end{cases}
\end{eqnarray}
where $\SC{L}(\E)$  is as in \eqref{def:LE}, and $I_P(\mu)$ is given by
$$I_P(\mu) =  -\inf_{f \in C_b(\E)^+} \int_E \log \f{Pf}{f}\ d\mu.$$

Consider  $M_n$ as an $\H^\#$-semimartingale, for $\H =M^\Phi(\pi) \subset L^\Phi(\pi)$ with $\Phi(x) = e^x -1$. Assume that $b$ is Lipschitz in the first argument and that 
$\sup_x\|b(x,\cdot)\|_\Phi < \infty$.\\ 
We now carry out the program of verification of LDP for $\{X_n\}$.\\

\np
{\it UET of $\{Y_n\}$:} Let $Z(u,s)$ be a cadlag process such that $\sup_{s\leq t}\|Z(\cdot,s)\|_\Phi \leq 1$.
Observe that without loss of generality we can take $Z \geq 0$ for our purpose.\\
Call  $H_n(t) \equiv  Z_-\cdot M_n (t)$. Then by It{\^o}'s lemma 
\begin{align*}
f(H_n(t))& =  f(H_n(0)) + \int_{\E\times \left[0,t\right]} f(H_n(s-) + Z(u,s-)) - f(H_n(s-))  \ M_n(du, ds).
\end{align*} 
Therefore, taking $f(x) = e^x$,
\begin{align*}
E(e^{H_n(t)})& = 1 + E\le(\int_{\E\times [0,t)} e^{H_n(s)}(e^{Z(u,s)} - 1)\mu_n(du,ds)\ri) \\
& = 1 + \sum_{k=1}^{\left[nt\right]} \int_\E E\le(e^{H_n(k/n)}(e^{Z(u,k/n)} - 1) P(\xi_k,du)\ri)\\
& \leq 1+ \sum_{k=1}^{\left[nt\right]}  e^{H_n(k/n)}\sup_{x}\int_\E (e^{Z(u,k/n)} - 1) P(x,du)\\
& = 1+ \sum_{k=1}^{\left[nt\right]}  e^{H_n(k/n)}\sup_{x}\int_\E (e^{Z(u,k/n)} - 1) \f{dP(x,\cdot)}{d\pi}(u)\pi(du)\\
& \leq  1+C\sum_{k=1}^{\left[nt\right]}  e^{H_n(k/n)}\int_\E (e^{Z(u,k/n)} - 1) \pi(du).
\end{align*} 
Now since $\|f\|_\Phi \leq 1 $ iff $\int\Phi(|f|)\ d\nu \leq 1$, we see from our assumption on the process $Z$ that
$\sup_{s\leq t}\int_E(e^{Z(u,s)} - 1) d\pi(u) \leq 1$. Thus,
\[E(e^{H_n(t)})\leq 1 + CE\int_0^t e^{H_n(s)}\nu_n(ds),\]
where $\nu_n\left[0,t\right] =\left[nt\right]$.
and by Gronwall's inequality
\[E(e^{H_n(t)}) = E(e^{Z_-\cdot M_n(t)}) \leq e^{C\left[nt\right]}.\]
Therefore
\begin{align*}
P(\sup_{s\leq t}n^{-1}Z_-\cdot M_n(s) > K )& =  P(Z_-\cdot M_n(t) > nK ) =P(e^{Z_-\cdot M_n(t)} > e^{nK} )\\
& \leq E(e^{Z_-\cdot M_n(t)}) /e^{nK} \leq e^{Cnt - nK}.
\end{align*}
Choosing $k(t,a) \equiv K= Ct+a$, we see that
\[\limsup_n\f{1}{n} \log\sup P\left[\sup_{s\leq t}n^{-1}|Z_-\cdot M_n(s)| > k(t,a)\right] \leq -a.\]\\
This proves that the sequence $\left\{Y_n \equiv M_n/n\right\}$ is UET.\\

\np
{\it LDP of $\{Y_n\}$:} Just as in Example \ref{Markovsde}, we have that 
for a pseudo-basis $\{(\phi_k,p_k)\}$ of $M^\Phi(\pi)$,  $\{Y_n(\alpha,\cdot) \equiv (Y_n(\phi_1,\cdot),Y_n(\phi_2,\cdot)),\cdots\}$ satisfies a LDP in $D_{R^\infty}[0,\infty)$with 
with the rate function given by
\begin{align}\label{rateY_occup}
	I_\alpha(y) = \inf\le\{\bar I(z): y_i(t) = \int_{\E\times [0,t]} \phi_i(u) z(du\times ds), \  z \in \SC{L}(\E), \ i=1,2,\hdots\ri\}, \quad  y  = (y_1,y_2,\hdots).
\end{align}	
Here $\bar I$ is defined in \eqref{rate_Markovocc}.\\

\np
{\it Exponential tightness of $\{X_n\}$}: By  It{\^o}'s lemma,
\begin{align*}
e^{nX^n(t+h)}& = e^{nX^n(t)} + \int_{\E\times[t,t+h)}e^{nX^n(s-)}(e^{b(X^n(s-),u)}-1)M_n(du\times ds).\\
\end{align*}
Recall that $\SC{F}^n_t = \SC{F}_{\left[nt\right]}$. It follows that,
$$E(e^{n(X^n(t+h)-X^n(t))}|\SC{F}^n_t) =1 + E\le( \int_{\E\times[t,t+h)}e^{n(X^n(s-) -X^n(t))}(e^{b(X^n(s-),u)}-1)\mu_n(du, ds)\bigg|\SC{F}^n_t\ri).$$
Hence 
\begin{align*}
E(e^{n(X^n(t+h)-X^n(t))}|\SC{F}^n_t)& \leq 1+  \sum_{nt\leq k < n(t+h))}\int_{\E}E\le(e^{n(X^n(k/n) -X^n(t))}(e^{b(X^n(k/n),u)}-1)P(\xi_{k-1},du)\big|\SC{F}_{\left[nt\right]}\right)\\
& \leq 1+ C\sum_{nt\leq k < n(t+h))}E \le(e^{n(X^n(k/n) -X^n(t))}\int_{\E}(e^{b(X^n(k/n),u)}-1)\pi(du)\Big|\SC{F}_{\left[nt\right]}\ri)\\
& \leq 1+ C \sup_x\|b(x,\cdot)\|_\Phi \sum_{nt\leq k <n(t+h))}E\le(e^{n(X^n(k/n) -X^n(t))}\big|\SC{F}_{\left[nt\right]}\ri).
\end{align*}
By Gronwall's inequality (Theorem 5.1, Page 498, \cite{EK86})
$$E\le(e^{n(X^n(t+h)-X^n(t))}\big|\SC{F}^n_t\ri) \leq e^{nC\sup_x\|b(x,\cdot)\|_\Phi h}$$
and as before,  exponential tightness of $\left\{X^n\right\}$ follows. \\

\np
{\it LDP and rate function of $\{X_n\}$}: Theorem \ref{sde} now readily establishes the LDP of $\left\{X^n\right\}$, and just like Example \ref{Markovsde},  the rate function  $J$ can be expressed as
$$J(x) = \inf \left\{ \int_0^\infty I_p(\dot{z}(s))ds: x(t) = x_0+ \int_{\E\times[0,t]}  b(x(s),u) z(du\times ds), \ z \in \SC{L}(E)\right\},$$
where $I_P$ is defined in \eqref{rate_Markovocc}.

}
\end{example}

\vspace{1cm}

Our next example is a Freidlin-Wetzell type small diffusion problem, where the driving integrator is space-time Gaussian white noise. A list of references for various LDP results on these types of models has already been mentioned in the introduction. Cho \cite{Cho06}  considers the case when the driving integrators are continuous orthogonal martingale random measures. 

\begin{example}[ Freidlin - Wentzell type LDP I ]\label{FWSDE1}{\rm 

Let $(E,r)$ be a complete and separable metric space and $\mu$ a sigma finite measure on $(E,\SC{B}(E))$. Let $\s: R^d\times E \rt \R^d $, be Lipschitz continuous in the sense that $\|\s(x,\cdot)  -\s(x',\cdot)\|_{L^2(\mu)} \leq L_\s|x-x'|$, and $L^2$-bounded, that is, $\|\s\|_\infty \equiv \sup_{x}\|\s(x,\cdot)\|_{L^2(\mu)} < \infty$. Suppose $b: \R^d \rt \R^d$ is a bounded Lipschitz function with bound $\|b\|_\infty = \sup_{x}|b(x)| < \infty$. Denote $W_n\equiv n^{-1/2}W$, where $W$ is the space time white noise on the measure space $(\E\times[0,\infty),\mu\ot\lambda_\infty)$. Recall that  $\lambda_\infty$ denotes the Lebesgue measure on $[0,\infty)$. Assume that $X_n$  satisfies 
$$X_n(t) = x_0  + \int_0^t b(X_n(s)) \ ds + \f{1}{\sqrt n} \int_{\E\times[0,t)}\s(X_n(s),u) W_n(ds\times du).$$\\

\np
{\it UET of $\{W_n\}$}:
Example \ref{ex_Gauss} shows that  $\left\{W_n\right\}$, indexed by $\H = L^2(\mu)$, is a sequence of uniformly exponentially tight $\H^\#$-semimartingales.\\

\np
{\it LDP of $\{W_n\}$}: As discussed in Example \ref{Gaussldp}, $\left\{n^{-1/2}W\right\}$ satisfies LDP,
and  if $\alpha=\left\{\phi_i\right\}$ forms a orthonormal basis, the rate function of $\left\{W_n(\alpha,\cdot)\equiv n^{-1/2}(W(\phi_1,\cdot),W(\phi_2,\cdot),\hdots)\right\}$ is given by 
\begin{eqnarray}\label{rateWn}
I_{\alpha}(y) &=
\begin{cases}
 \f{1}{2}\sum_{i=1}^\infty \int_0^\infty  |\dot{y}_i(t)|^2 \ dt,& \ \  \ y_i(t) = \int_0^t \dot{y}_i(s) \ ds \mbox{ for some } \dot{y}_i \in L^2(\R) \\
 \infty & \ \ \mbox{otherwise}.
 \end{cases}
 \end{eqnarray}

\np
{\it Exponential tightness of $\{X_n\}$}:  For simplicity of calculation, we assume $d=1$. But the following steps could be easily extended to $d>0$ (for example, by applying them component-wise). To show that $\left\{X_n\right\}$ is exponentially tight, observe that
by It\^{o}'s lemma,
\begin{align*}
 e^{nX_n(t+h)} = e^{nX_n(t)} + \int_t^{t+h} ne^{nX_n(s)} dX_n(s) + \f{1}{2}\int_t^{t+h} n^2e^{nX_n(s)} d\left[X_n\right]_s.
\end{align*}
It follows that 
\begin{align*}
 E\le(e^{n(X_n(t+h) - X_n(t))}\big| \SC{F}^n_t\ri)& = 1 + \f{1}{2}\int_{0}^{h} nE\le(e^{n(X_n(t+s) - X_n(t))} (\|\s(X_n(s),\cdot)\|_{L^2(\mu)}^2 + b(X_n(s))\Big | \SC{F}^n_t\ri)\  ds\\
& \leq 1 + \f{n(\|\s\|^2_\infty+\|b\|_\infty)}{2}\int_{0}^{h} E\le(e^{X_n(t+s) - X_n(t)}\big| \SC{F}^n_t\ri) \ ds.
\end{align*}
 Hence by Gronwall's inequality,
$$E\le(e^{n(X_n(t+h) - X_n(t))}\big| \SC{F}^n_t\ri)  \leq e^{n(\|F\|^2 + \|b\|)h}.$$
Similarly,
$$E\le(e^{-n(X_n(t+h)-X_n(t))}\big|\SC{F}^n_t\ri) \leq e^{n(\|F\|^2 + \|b\|)h},$$
and it follows that  
$$E\le(e^{n|X_n(t+h)-X_n(t)|}\big|\SC{F}^n_t\ri) \leq 2e^{n(\|F\|^2 + \|b\|)h}.$$
As before, Theorem 4.1 of Feng and Kurtz \cite{FK06} implies that $\left\{X_n\right\}$ is exponentially tight. \\

\np
{\it LDP and rate function of $\{X_n\}$}:  Put $\y(t) = \sum_k y_k(t)\phi_k$, where  $\left\{\phi_k\right\}$ is a chosen  complete orthonormal system of $L^2(\mu)$. Theorem \ref{sde} implies that $\left\{(X_n, n^{-1/2}W(\alpha,\cdot)\right\}$  satisfies a LDP in $C_{\R^d\times \R^\infty}[0,\infty)$ with the rate function 
\begin{eqnarray}
J_{\alpha}(x, y)& = 
\begin{cases}
I_{\alpha}(y), & x(t) = x_0 + \int_0^t b(x(s))ds + \s(x,\cdot)\cdot \y(t), \ \ \  \y \in\SC{D},\\
\infty, & \mbox{otherwise.}
\end{cases}
\end{eqnarray}
Letting $\psi(u,t) = \sum_{k} \dot y_k(t)\phi_k(u)$, observe that $ \s(x,\cdot)\cdot \y(t) = \int_{\E\times [0,t]}  \s(x(s),u) \psi(u,s) \mu(du) ds$ and 
$\int_{\E\times [0,\infty]}|\psi(u,t)|^2\mu(du)dt = \sum_{i=1}^\infty \int_0^\infty  |\dot{y}_i(t)|^2 \ dt.$ Consequently, it is easy to conclude that $X_n$ satisfies a LDP in  $C_{\R^d}[0,\infty)$ with the rate function given by
 \begin{align*}
 	J(x) =& \inf\bigg\{\f{1}{2}\int_{\E\times [0,\infty]}|\psi(u,t)|^2\mu(du)dt: x(t) = x_0 + \int_0^t b(x(s))ds+  \int_{\E\times [0,t]}  \s(x(s),u) \psi(u,s) \mu(du) ds,\\
 &	\hs*{1cm} \psi \in L^2(E\times[0,\infty))\bigg\}.
 \end{align*}

}
\end{example}

\begin{example} [\bf Freidlin - Wentzell type LDP II] 
\label{FWSDE2}
{\rm

 Consider the SDE
  \begin{align}
  \label{Jumpdiff}
X_n(t)& = x_0 + n^{-1/2}\int_{\E_1\times[0,t]}\sigma_1(X_n(s),u) W(ds\times du) + \int_0^t b(X_n(s)) \ ds\\
& \ \ \  + n^{-1}\int_{\E_2\times[0,t]}\sigma_2(X_n(s),v) \xi(n \ ds\times dv).
\end{align}
Here, as before, $W$ is a space time white noise on $(\E_1\times[0,\infty),\mu\ot\lambda_\infty)$, $\xi$ is a Poisson random measure on $\E_2\times[0,\infty)$ with mean measure $\nu\ot\l_\infty$, and is independent of  space-time white noise $W$. Assume that $\s_1$ and $b$ satisfy the same conditions $\s$ and $b$ did in Example \ref{FWSDE1}. For $\sigma_2: R^d\times E_2 \rt \R^d $, assume that 
$$|\s_2(x,v)| \leq M_2(v), \quad  |\s_2(x,v) - \s_2(x',v)| \leq L_2(v)|x - x'|$$
 where $M_2$ and $L_2$ are in $M^{\Phi}(\nu)$, with $\Phi(x) = e^x -1$ (see (\ref{MTspace})). In other words, we assume that
 $$\int_{E_2}\le(e^{aM_2(v)} - 1\ri)\nu(dv) < \infty, \ \mbox{ and  }\  \int_{E_2}\le(e^{aL_2(v)} - 1\ri)\nu(dv) < \infty, \quad \mbox{ for all } a>0.$$
 
  Put $Y_n =(n^{-1/2}W, n^{-1}\xi_n)$ with $\xi_n(A\times[0,t]) = \xi(A\times[0,nt])$ and take the indexing space $\H =(L^2(\mu),M^{\Phi}(\nu) ) $ with $\|\cdot\|_\H = \|\cdot\|_{L^2(\mu)} + \|\cdot\|_\Phi$. Examples \ref{ex_Gauss} and \ref{ex_Poisson} establish that $Y_n$ is UET.\\

\np
{\it LDP of $\{Y_n\}$}: Let $\{\phi_k^1\}$ be a complete orthonormal system of $L^2(\mu)$ and $\{(\phi_k^2, p_k^2)\}$ a pseudo-basis of $M^{\Phi}(\nu)$.  By Example \ref{poildp}, $\left\{n^{-1}\xi_n(\alpha^2,\cdot) \equiv n^{-1}(\xi_n(\phi_1^2,\cdot), \xi_n(\phi_2^2,\cdot), \hdots)\ri\}$ satisfies a LDP in $D_{\R^\infty}[0,T]$ with the rate function given by $\tilde I_{\alpha^2}$
\begin{align*}
 \tilde I_{\alpha_2}(y) = \inf\le\{L_T(\varphi): y_i(t) = \int_{\E_2\times [0,T]} h_i(v)\varphi(z,s)\nu(dv)ds, \ i=1,2,\hdots,\ri\}
\end{align*}
where $L_T$ is defined in \eqref{LT_poi}. Also, $\left\{W_n(\alpha^1,\cdot) \equiv (W_n(\phi_1^2,\cdot), W_n(\phi_2^2,\cdot), \hdots)\ri\}$ satisfies a LDP in $C_{\R^\infty}[0,T]$ with rate function $I_{\alpha_1}$ defined by \eqref{rateWn} (with the small notational change of $\alpha$ to $\alpha_1$ and integration ranging from $0$ to $T$). And thus by independence of $W$ and $\xi$, $\le\{Y_n(\alpha^1, \alpha^2, \cdot) = \le(W_n(\alpha^1,\cdot), n^{-1}\xi_n(\alpha^2,\cdot)\ri)\ri\}$ satisfies a LDP in $C_{\R^\infty}[0,T]\times D_{\R^\infty}[0,T]$ with the rate function given by $I_{\alpha_1}+ \tilde I_{\alpha_2}$.\\

%
%
%
\np
{\it Exponential tightness of $\{X_n\}$}: Similar to Examples \ref{Markovsde} and \ref{FWSDE1}, It{\^o}'s formula and Gronwall's inequality  prove that $E(\exp(n|X_n(t+h)-X_n(t)|)) = O(e^{Cnh})$. This verifies the exponential tightness of the solution. As before, Theorem \ref{sde2} gives the associated rate function for $\left\{X_n\right\}$.\\

\np
{\it LDP and rate function of $\{X_n\}$}:   An application of Theorem \ref{sde} implies that $\{X_n\}$ satisfies a LDP and it could be easily checked, after a few simple steps, that the associated rate function is given by
\begin{align*}
 	J(x) =&\inf \bigg\{L_T(\varphi)+\f{1}{2}\int_{ \E\times [0,T]}|\psi(u,t)|^2\mu(du)dt: x(t) = x_0 + \int_0^t b(x(s))ds\\
& \hs*{1cm}	+  \int_{\E_1\times [0,t]}  \s_1(x(s),u) \psi(u,s) \mu(du) ds
 + \int_{\E_2\times [0,t]}  \s_2(x(s),v) \varphi(v,s) \nu(dv) ds, \\
 & \hs*{1cm} \psi \in L^2(E\times[0,T]),\ L_T(\varphi)<\infty\bigg\}.
 \end{align*}

Some of the conditions like boundedness of $b$, $\s_1$ and $\s_2$ (in the above sense) made  the proof of exponential tightness of $\{X_n\}$ simpler, but with a little extra work they could be relaxed to that having linear growth.

}
\end{example}

\begin{example}[Two-scale hybrid diffusion process] {\rm We now consider a hybrid diffusion processes of the form:
$$X_n(t)  =x_0 + \int_{0}^t b(X_n(s), Y_n(s)) ds + \f{1}{\sqrt n}\int_{\E \times [0,t]}\s (X_n(s),u)W(du\times ds).$$
Here $W$ is the space time white noise on the measure space $(\E\times[0,\infty),\mu\ot\lambda_\infty)$, $\l_\infty$ is the Lebesgue measure on $[0,\infty)$, $\mu$ is a $\s$-finite measure on $E$, $Y_n(t) =  Y(nt)$, where $Y$ is an ergodic Markov process on a compact metric space $\U$ with the unique invariant measure $\pi$ and which is independent of $W$. These types of processes are characterized by property that the dynamics of the slow diffusion $X_n$  is modulated by the fast moving Markov process $Y_n$ \cite{YinZhu10}. Under some standard assumptions on the coefficients, it is easy to see that an averaging principle holds: that is $X_n \rt X$ where
$$X(t) = x_0 + \int_0^t \bar b(X(s))\ ds,$$
where $\bar b(x) = \int_{\U} b(x,v)\pi(dv).$ We are interested in a LDP for $X_n$ and we briefly describe how the general program of verification of LDP can be used for this purpose. Some problems related to LDP for these types of systems, where the slow process has no diffusion term, have been considered in  \cite{FW98, HYZ11}. We assume that $b: \R^d\times \U \rt \R^d$ is bounded and Lipschitz in the first component, $\s: \R^d\times \E \rt \R^d$ is Lipschitz continuous in the sense that $\|\s(x,\cdot)  -\s(x',\cdot)\|_{L^2(\mu)} \leq L_\s|x-x'|$, and $L^2(\mu)$-bounded, that is, $\sup_{x}\|\s(x,\cdot)\|_{L^2(\mu)} <\infty.$ Next, we assume that the sequence of occupation measures $\G_n$ defined by $\G_n(A\times[0,t]) =\f{1}{n}\int_0^{nt}1_{\{Y(s)\in A\}}\ ds = \int_0^t 1_{\{Y_n(s)\in A\}}\ ds$ satisfies a LDP in $C_{\SC{M}_F(\U)}[0,\infty)$ with a rate function $\bar I$. Typically, in many examples (see \cite{FK06}) $\bar I$ is given by
\begin{eqnarray} \label{rate_Markovocc2}
	\bar I(z)& = 
	\begin{cases}
		\int_0^\infty I_{\SC{A}}(\dot{z}(s))\ ds, & z \in \SC{L}(\U),\\
		\infty, & \mbox{otherwise,}
	\end{cases}
\end{eqnarray}
where $I_{\SC{A}}(\nu)$ is given by
$$I_{\SC{A}}(\nu) =  -\inf_{f \in D(\SC{A})\cap C_b(\U)^+} \int_\U  \f{\SC{A} f}{f}\ d\nu.$$ 
Here $\SC{A}$ is the generator of the Markov process $Y$, $D(\SC{A})$ is the domain of $\SC{A}$ and  $\SC{L}(\U)$ is as in Example \ref{Markovsde}.

Note that $X_n$ satisfies
$$X_n(t)  =x_0 + \int_{\U \times [0,t]} b(X_n(s), v) \G_n(dv\times ds) + \f{1}{\sqrt n}\int_{\E \times [0,t]} \s (X_n(s),u)W(du\times ds).$$

Let $C(\U)$ be equipped with the sup norm: $\|f\|_{\infty}  =\sup_{v\in \U}|f(v)|$. Put $Y_n =(\G_n, n^{-1/2}W)$ and take the indexing space $\H =(C(\U),L^2(\mu) ) $ with $\|(f,h)\|_\H = \|f\|_{\infty} + \|h\|_{L^2(\mu)}$.  Example \ref{ex_Gauss} established that $n^{-1/2}W$ is UET. It is also easy to see that $\G_n$ is UET. Indeed, let $Z$ be a cadlag process taking values in $C(\U)$ such that $\sup_{s\leq t} \|Z(\cdot,s)\|_{\infty} \leq 1$. Then $\sup_{r\leq t} |Z\cdot \G_n (r)| = \sup_{r\leq t} \le|\int_0^rZ (Y_n(s),s) ds\ri| \leq t$, and it follows that
$$\limsup_n\f{1}{n} \log\sup P\left[\sup_{s\leq t}|Z\cdot \G_n(s)| > k(t,a)\right] \leq -a.$$
with $k(t,a) = t+a$.
This proves that the sequence $\left\{\G_n\right\}$, and consequently, $\{Y_n =(\G_n, n^{-1/2}W)\}$ is UET. Exponential tightness of $X_n$ follows by calculations similar to that in Example \ref{FWSDE1}, and like Examples \ref{rndev}  and    \ref{FWSDE1}, we  conclude that $X_n$ satisfies a LDP in $C_{\R^d}[0,\infty)$ with the rate function given by
\begin{align*}
J(x) = & \inf \left\{ \bar I(z) +\f{1}{2} \int_{\E\times[0,\infty)} |\psi(u,s)|^2\mu(du)ds : x(t) = x_0+ \int_{\U\times[0,t]}  b(x(s),v) z(dz\times ds) \ri.\\
&\le. \hs*{1cm} +\int_{\E\times[0,t]} \s(x(s),u)\psi(u,s)\mu(du) ds, \ \psi \in L^2(\E\times[0,\infty)), \ z \in \SC{L}(\U)\ri\}.
\end{align*}
Some of the conditions like boundedness of $b$ and $\s$ could be relaxed to that having linear growth without too much difficulty. Also, the program could be adopted to cover the case of $Y$ taking values in a non-compact $\U$ and $\s$ depending on the states of both the slow process $X_n$ and fast process $Y_n$. However, the corresponding analysis requires more careful estimates and deserves a separate paper-long treatment.

}	
\end{example}

\setcounter{section}{0}
\setcounter{theorem}{0}
\setcounter{equation}{0}
\renewcommand{\theequation}{\thesection.\arabic{equation}}

\appendix
\section*{Appendix}
\renewcommand{\thesection}{A} 

\subsection{Some definitions}

\begin{definition}\label{def:exptight}
Let $U$ be a Polsh space and $\left\{\mu_n\right\}$  a sequence of probability measures on $(U, \SC{U})$, where $\SC{U}$ is the Borel $\sigma$-algebra on $U$ containing all the compact subsets of $U$.
$\left\{\mu_n\right\}$ is {\bf exponentially tight} if for every $a > 0$, there exists a compact set $K_a$ such that
$$\limsup_{n\rt \infty} \frac{1}{n} \log \mu_n(K_a^c) \leq -a.$$
\end{definition}

\begin{definition}\label{def:expcpt}
Let $E$ be a complete and separable metric space. A sequence $\{X_{n}\}$ satisfies the {\bf exponential compact containment condition} if for each $a,T>0$, there exists a compact set $C_{a,T} \subset E$ such that
$$\limsup_{n\rt\infty}\frac{1}{n} \log P(X_{n}(t) \notin C_{a,T} \ \mbox{ for some }  t\leq T) \leq -a.$$
\end{definition}
Clearly, exponential tightness implies exponential compact containment condition and the latter implies the former under the additional requirement of exponential tightness of  the sequence of real valued processes $\{f(X_n)\}$, for every $f$ belonging to an appropriate function family \cite[Theorem 4.4]{FK06}. 

\subsection{A generalized contraction principle}
\begin{theorem}
\label{cont_mod}
Let $(E,r)$ and $(E',r')$ be two complete, separable metric spaces. Let $\left\{X_n\right\}$ be a sequence of random vectors taking values in  $E$. Suppose that $\left\{X_n\right\}$satisfies a large deviation principle with the good rate function $I$. Assume that $f,f_n: E\rt E'$ are measurable functions satisfying:
\begin{itemize}
\item for all $x \in E$ with $I(x)<\infty$, $x_n\rt x$ implies that $f_n(x_n)\rt f(x).$
\end{itemize}
Then $\left\{f_n(X_n)\right\}$ satisfies a large deviation principle with the rate function given by 
$$I'(y) =\inf\left\{I(x):f(x)=y\right\}.$$
\end{theorem}
See \cite[Theorem 2.4]{Gar04}.

\subsection{Orlicz spaces}
\label{Orlicz}
The standard reference for this section is Rao and Ren \cite{RR91} . Some results presented here are taken from Terrence Tao's lecture notes on Harmonic Analysis \cite{Tao247A}.\\
Let $U$ be a complete and separable metric space, and $ \SC{U}$ a $\sigma$-algebra on $U$. 
Observe that for the space $L^p(U,\mu), 1\leq p<\infty$,
\[\|f\|_p \leq 1 \mbox{ iff } \int_U |f|^p \ d\mu \leq 1. \]
The motivation for introducing Orlicz spaces is to find more general function $\Phi: \R \rt \R^+$ satisfying certain conditions, such that the above kind of statement is true, that is we want to find a norm
$\|f\|_\Phi$ such that
\[\|f\|_\Phi \leq 1 \mbox{ iff } \int_U \phi(|f|) \ d\mu \leq 1. \]

\begin{definition}
 Let $\Phi: \R \rt \R^+$ be an even, increasing and convex function with $\Phi(0) = 0$ and $\lim_{x\rt \infty} \Phi(x) =\infty.$ Such a $\Phi$ is called a {\bf Young's function}.

Define the norm $\|\cdot\|_{\Phi}$ by
\[\|f\|_\Phi \equiv \inf\left\{A>0: \int_U \Phi(|f|/A) \ d\mu \leq 1\right\}.\]
 $\|\cdot\|_{\Phi}$ is called  the {\bf Orlicz norm}, and the corresponding space
\[L^\Phi(U,\mu) \equiv \left\{f:\|f\|_\Phi < \infty\right\} \]
is called the {\bf Orlicz space}.
\end{definition}

\begin{lemma}
 Orlicz spaces are Banach spaces.
\end{lemma}

\np
The following are a few examples of Orlicz spaces. 
\begin{itemize}
 \item $L^p(U,\mu)$ forms an Orlicz space for $1\leq p <\infty$, with $\Phi(x) = |x|^p$.
 \item The spaces $L^\Phi(U,\mu)$ with $\Phi(x)\equiv e^x - 1 $, or $\Phi(x) = x\log(x+2)$.
\end{itemize}
%
Observe that 
$$L^\Phi(U,\mu)  = \left\{f: \int \Phi(af) \ d\mu < \infty, \ \mbox{ for some } \ a>0\right\}.$$
Let 
\begin{align}
 \label{MTspace}
M^\Phi(U,\mu) = \left\{f: \int \Phi(af) \ d\mu < \infty, \ \mbox{ for all } \ a>0\right\}.
\end{align}
The space $M^\Phi(U,\mu)$ was introduced by Morse and Transue (1950), and is sometimes referred to as \textbf{Morse-Transue space} \cite{RR91}.

\begin{lemma}
Let $\Phi$ be a continuous Young's function. Then $M^\Phi(U,\mu)$ is a closed linear subspace of $L^\Phi(U,\mu) $.
\end{lemma}

In general,  many interesting Orlicz spaces might not  be separable, for example, $L^\Phi(U,\mu)$ with $\Phi(x) =e^x-1$ is not separable. However, for $M^\Phi$, we have the following theorem (Page 87, \cite{RR91}):

\begin{lemma}
Let $(U,\SC{U})$ be a complete and separable measure space, and $\Phi$ a continuous Young's function with $\Phi(x) =0$ iff $x =0$. Then the space $M^\Phi(U,\mu)$ is separable.
\end{lemma}

\subsection{Basis theory}
\label{Schauder}
The material presented here is taken from \cite{Heil97, Singer70, Singer81}.

\begin{definition}
 A sequence $\left\{x_k\right\}$ in a Banach space $B$ is a  {\bf basis} for $B$ if for each $x\in B$, there exist unique scalars $p_k(x)$, such that
\[x = \sum_k p_k(x) x_k.\]
\end{definition}

\begin{remark}
Every Banach space with a basis is separable.
\end{remark}

\begin{remark}
It is easy to see that the $p_n$ are linear functionals.
\end{remark}

\begin{definition}\label{Schau}
A basis $\left\{x_k\right\}$ is called a {\bf Schauder basis} if the unique $p_k$ are bounded linear functionals, that is if $p_k \in X^*$ for every $n$.
\end{definition}

\begin{theorem}
Every basis $\left\{x_k\right\}$ of $X$ is a Schauder basis, that is, the $p_k$ are bounded linear functionals.
\end{theorem}

\begin{example}
Every separable Orlicz space (hence $L^p$ space) has a Schauder basis. Every separable Hilbert space has a Schauder basis given by its complete orthonormal system.
\end{example}

\np
The notion of a basis of a Banach space is generalized to that of {\it pseudo-basis} defined below. 
\begin{definition}
\label{pseudo_def}
A sequence $\left\{x_k\right\}$ in a Banach space $B$ with $x_k \neq 0$ for $k=1,2,\hdots$ is a {\bf pseudo-basis} if for every $x \in B$, there exists a sequence of scalars $\left\{p_k\right\} $ such that
\begin{align}
\label{pseudo}
x = \sum_k p_k x_k.
\end{align}
\end{definition}

\begin{theorem} \cite[Theorem 5.1]{Singer81}
\label{hom}
Let $B$ be a separable Banach space.
\begin{enumerate}[label={\rm (\roman*)}, leftmargin=*, align=right]
\item Then $B$ has a pseudo-basis. 
\item \label{2} Every sequence $\left\{x_k\right\}$ with $x_k\neq 0, k=1,\hdots,$  which is dense in $\left\{x \in B: \|x\| \leq 1\right\}$ is a pseudo-basis of $B$. For every such sequence $\left\{x_k\right\}$, there exists a subset $\SC{L} $ of $l^1$ with the following property: for every $x \in B$ there exists a unique sequence of scalars $\left\{p_k(x)\right\}$ such that (\ref{pseudo}) is satisfied and the mapping
$$x \rt \left\{p_k(x)\right\}$$
is a homeomorphism of $E$ onto $\SC{L}$.
\end{enumerate}
\end{theorem}

\begin{remark}{\rm
In partcular, the above theorem implies that for each $k$, the mapping $x \rt p_k(x)$ is continuous. If $\left\{x_k\right\}$ is a basis, then the $p_k$ are also linear, hence $p_k \in B^*, k = 1,2,\hdots$.}
\end{remark}

\np
{\bf Notation:} For convenience, we denote a basis or a pseudo-basis of $B$ by $\left\{(x_k, p_k)\right\}$.\\

\subsection{A compactness lemma}
\begin{lemma}
\label{cptconv}
Let $B$ be a separable Banach space with pseudo-basis $\left\{(x_k,p_k)\right\}$ satisfying \ref{2} of Theorem \ref{hom}. Define a sequence of  continuous functions $\left\{S_n\right\}$ by
$$S_n(x) = \sum_{k=1}^np_k(x)x_n.$$
Then $S_n \rt I$ uniformly on compacts, that is, for every compact set $C \subset B$
$$\sup_{x\in C} \|S_n(x) - x\| \rt 0.$$
\end{lemma}

\begin{proof}
Let $C\subset B$ be compact. Let $T$ denote the mapping
$$x \in B \rt \left\{p_k(x)\right\} \in \SC{L}.$$
Note that by Theorem \ref{hom}, $T(C) \subset \SC{L}$ is also compact. Fix an $\e>0$. Define the  open sets $O_n \subset \SC{L}$ by
\begin{align}
\label{open}
O_n = \left\{\left\{c_k\right\} \in l^1: \sum_{j=n}^\infty |c_j| <\e\right\}.
\end{align}
Notice that the $O_n$ are increasing and $T(C)\subset \cup_{n}O_n$. Since $T(C)$ is compact, there exists an $N>0$, such that
$T(C) \subset \cup_{j=1}^NO_j = O_N.$ It follows using (\ref{2}) of Theorem \ref{hom} that if $n>N$, then $\|S_n(x) - x\| < \e$, for all $x\in C$.
\end{proof}

\begin{remark} {\rm If $B$ has a Schauder basis $\left\{(x_k,p_k)\right\}$, then the above conclusion holds as well. This can be seen by an application of Arzela-Ascoli theorem. The proof needs to be different because a Schauder basis  $\left\{(x_k,p_k)\right\}$ might not satisfy \ref{2} of Theorem \ref{hom}.} 
\end{remark}

\subsection{Integration with respect to vector-valued functions}
\label{vectint}
Suppose $\X$ is a  Banach space, and $x \in D_\X[0,\infty)$. Suppose $\y \in D_{\X^*_c}[0,\infty)$ is of finite variation in the sense that  $T_t(\y) < \infty$, for all $t>0$, where  the total variation $T_t(\y)$ is  defined as 
$$T_t (\y) = \sup_{\sigma} \sum_i \|\y(t_i) -\y(t_{i-1})\|_{\X^*},$$
 $\sigma\equiv\left\{t_i\right\}_i$ varying over all partitions of $[0,t)$.\\

\np
Define the integral $x\cdot \y$ by
\begin{align}
\label{intdef}
x\cdot\y(t)& = \lim_{\|\sigma\|\rt 0} \sum_i \<x(t_i), \y(t_{i+1}) -\y(t_i)\>_{\X,\X^*},
\end{align}
$\|\sigma\|$ denoting the mesh of the partition $\sigma\equiv\left\{t_i\right\}_i$. Here $\<h,h^*\>_{\X,\X^*} = h^*(h)$ for $h\in \X, h^* \in \X^*$.

\begin{lemma}
\label{intpf}
The limit in (\ref{intdef}) exists.
\end{lemma}
\begin{proof}
For  $\sigma\equiv\left\{t_i\right\}_i$, denote 
\[x^\sigma(t) = \sum_{i} x(t_i)1_{[t_i,t_{i+1})}(t).\]
And notice that for a finer partition $\delta$,
\[\sup_{r\leq t}|(x^\sigma-x^\delta)\cdot \y(r)| \leq \int_0^t\|x^\sigma(s)-x^\delta(s)\|_{\X}dT_s(\y) \leq \sup_{s\leq t}\|x^\sigma(s)-x^\delta(s)\|_{\X}T_t(\y).\]
It follows that $\left\{x^\sigma\cdot \y(t)\right\}$ is a Cauchy sequence and we are done.
\end{proof}

More generally, we can allow the integrands to take values in some operator space, so that the integral is infinite-dimensional. Let $\Y$ be a Banach space, and suppose that $x \in D_{L(\X^*,\Y)}[0,\infty)$. Define the integral $x\cdot \y$ by
\begin{align}
\label{intdef2}
x\cdot\y(t)& = \lim_{\|\sigma\|\rt 0} \sum_i x(t_i)\circ (\y(t_{i+1}) -\y(t_i)),
\end{align}
$\|\sigma\|$ denoting the mesh of the partition $\sigma\equiv\left\{t_i\right\}_i$. Here for $S \in L(\X^*,\Y) $ and $x \in \X^*$, $S \circ x  =S(x)$. The proof of the existence of the limit is same as that of Lemma \ref{intpf}. Notice that $x\cdot \y $ takes values in $\Y$. 
We end by noting that the above integrals are just  special (deterministic) cases of integrals with respect to $(\LL,\hat{\H})^\#$-semimartingales.

\vspace{1cm}
\np
\textbf{Acknowledgement:}
I am grateful to my advisor Prof. Tom Kurtz for his numerous advice and comments throughout the preparation of the paper.

\bibliographystyle{plainnat}
\bibliography{Grant}
\end{document}